\documentclass[11pt,reqno]{amsart}
\usepackage{amsmath,amscd,amssymb,amsfonts,amsthm,courier,relsize,bm}
\usepackage{hyperref,enumerate,mathrsfs,mathtools,slashed}
\usepackage[nobysame]{amsrefs}
\usepackage{tikz-cd}
\usepackage{marginnote}

\usepackage{xcolor}  %the next few lines set the colors of the hyperlinks, and remove the color boxes
\hypersetup{
    colorlinks,
    linkcolor={red!50!black},
    citecolor={blue!70!black},
    urlcolor={blue!80!black}
}

\newtheorem{theorem}{Theorem}[section]
\newtheorem{corollary}[theorem]{Corollary}
\newtheorem{proposition}[theorem]{Proposition}

\newtheorem{lemma}[theorem]{Lemma}
\theoremstyle{definition}    
\newtheorem{definition}[theorem]{Definition}
\theoremstyle{remark}

\newtheorem{remark}[theorem]{Remark}

\newtheorem{example}[theorem]{Example}

\newcommand{\ignore}[1]{}

\renewcommand{\i}{{\mathrm{i}}}

\def\bC{\ensuremath{\mathbb{C}}}
\def\bR{\ensuremath{\mathbb{R}}}
\def\bZ{\ensuremath{\mathbb{Z}}}

\def\bT{\ensuremath{\mathbb{T}}}

\def\ker{\ensuremath{\textnormal{ker}}}

\def\DNC{\ensuremath{\textnormal{DNC}}}

\def\tr{\ensuremath{\textnormal{tr}}}

\def\dim{\ensuremath{\textnormal{dim}}}

\def\Str{\ensuremath{\textnormal{Str}}}
\def\str{\ensuremath{\textnormal{str}}}

\def\Ch{\ensuremath{\textnormal{Ch}}}

\def\Ahat{\ensuremath{\widehat{\textnormal{A}}}}

\def\act{\ensuremath{\curvearrowright}}

\newcommand{\br}[1]{\bm{\mathrm{#1}}}

\title{Deformation Spaces, Rescaled Bundles and the Kirillov Character Formula}
\author{Maxim Braverman and Ahmad Reza Haj Saeedi Sadegh }
\date{November 2022}

\begin{document}
\begin{abstract}
    In this paper, we construct a smooth vector bundle over the deformation to the normal cone $\DNC(V,M)$ through a rescaling of a vector bundle $E\to V$, which generalizes the construction of the spinor rescaled bundle over the tangent groupoid by Nigel Higson and Zelin Yi. We also provide an equivariant version of their construction. As the main application,  we recover the Kirillov character formula for the equivariant index of Dirac-type operators. As another application, we get an equivariant generalization of the description of  the Witten and the Novikov deformations of the de Rham-Dirac operator using the deformation to the normal cone obtained recently by O.~Mohsen. 
\end{abstract}
\maketitle

%introduction
\section{Introduction}
In the coadjoint orbit method of Kirillov, every character of a compact Lie group is given as an integral over some orbits of the coadjoint action of $G$ on $\mathfrak{g}^*$. This is known as the {\em Kirillov character formula}. Berline and Vergne \cite{berline1985equivariant} (see also \cite[Ch.~8]{berline1992heat}) extended this formula to a {\em delocalized} equivariant index theorem for a general equivariant Dirac-type operator on a closed manifold $M$.  In this version of the index theorem, also known as {\em the (generalized) Kirillov formula}, the equivariant index is expressed as an integral of an equivariant differential form  over the whole manifold. 
Applying  the localization formula for  the integral of an equivariant form to the Kirillov formula, one recovers  Atiyah-Segal fixed point formula for the index, \cite{atiyah1968index}.

 In \cite[Ch.~8]{berline1992heat}, a local version of the Kirillov form formula is proven by adopting Bismut's generalization of the Getzler rescaling  technique (see \cite{bismut1985index}). In this paper, we give  a  coordinate-free version of the proof in \cite{berline1992heat} by using the tangent groupoid $\mathbb{T}M$.  This gives an equivariant generalization of the method of  Higson and Yi, \cite{Higson2019SpinorsAT}. 
 
We also study the \emph{deformation to the normal cone} $\DNC(V,M)$ for an embedding $M\hookrightarrow V$ (see Subsection \ref{ss:deformationtothenormalcone}, or \cite{debord2019lie,rouse2008schwartz,sadegh2018euler}), and generalize the rescaled bundle construction of \cite{Higson2019SpinorsAT} to this deformation space. In some sense, this generalization makes the exposition clearer and more conceptual since it makes it evident which structures are needed for the construction to work. We also describe an equivariant version of this construction. As one of the applications, we provide an equivariant generalization of  the description of the Witten deformation via the deformation to the normal cone, obtained in a recent paper by Mohsen, \cite{Mohsen22}, cf. Subsection~\ref{ex:triviallyfilteredcase}. We also generalize this construction to the Novikov deformation, \cite{Pazhitnov87,BrFar1,BrFar3}.

We  now give a slightly more detailed review of this construction. 

To an embedding of smooth manifolds $M\hookrightarrow V$, one may associate the {deformation to the normal cone}, denoted $\DNC(V,M)$, as a set is given by the disjoint union  

\[\mathcal{N}\times\{0\}\bigsqcup_{t\neq 0}V\times\{t\}\]
where $\mathcal{N}$ is the normal bundle of the embedding. There is a smooth structure on  $\DNC(V,M)$ which can be given by an explicit atlas (\cite{rouse2008schwartz}) or can be given using an algebraic point of view (see Subsection \ref{ss:deformationtothenormalcone} or \cite{sadegh2018euler}).

One important example of the deformation to the normal cone is the \emph{tangent groupoid}. Associated with the diagonal embedding $M\hookrightarrow M\times M$, the deformation space $\DNC(M\times M,M)$, also denoted by $\mathbb{T}M$, is called the tangent groupoid of $M$. As a set, it is given by the disjoint union 

\[TM\times\{0\}\bigsqcup_{t\neq 0}M\times M\times\{t\}.\]
In \cite{Connes94noncommutativegeometry}, Connes used the tangent groupoid whose K-theory carries the analytical index of elliptic operators. In \cite{van2015groupoid}, van Erp and Yuncken used the tangent groupoid to completely characterize the algebra of pseudodifferential operators. 

In \cite{Higson2019SpinorsAT}, Higson and Yi introduced a coordinate-free approach to Getzler's rescaling method. This was achieved by introducing the \emph{rescaled spinor bundle} over the tangent groupoid of a spin manifold. So for a spin manifold $M$, with the spinor bundle ${S}\to M$, Higson and Yi,  \cite{Higson2019SpinorsAT}, constructed a vector bundle $\mathbb{S}\to \mathbb{T}M$ over the tangent groupoid. This bundle is given by the following diagram
\[\begin{tikzcd}
\pi^*\Lambda^*T^*M  \arrow[d, ]
&& S\boxtimes S^* \arrow[d, ] \\
TM\times\{0\} &\bigsqcup_{t\neq 0}
& M\times M\times\{t\}.
\end{tikzcd}\]

In this paper, we generalized the construction of the rescaled bundle in \cite{Higson2019SpinorsAT} to deformation to the normal cones and gave some examples of the rescaled bundles. As an application, we give a proof of the Kirillov formula by introducing a family of rescaled bundles over the tangent groupoid of a manifold $M$ that carries a $G$-equivariant Clifford module.

The paper is organized as follows: In Section \ref{s:equivariancohomolgy}, we give a quick overview of Cartan's model of the equivariant cohomology and equivariant characteristic classes.

In Section \ref{s:equivariantindexandkirillovformula}, we state the (generalized) Kirillov formula and give an overview of the heat kernel proof of this formula. In particular, we introduce the  Bismut Laplacian and explain its role in the proof.

In Section \ref{s:rescaledbundles}, we introduce the deformation to the normal cone $\DNC(V,M)$ associated with an embedding of manifolds $M\hookrightarrow V$. If there is a vector bundle $E\to V$ with connection compatible with certain filtrations on $E|_M$ and on the bundle of endomorphism $\textup{End}(E)$, then we introduce a vector bundle $\mathbb{E}\to \DNC(V,M)$. In the construction of the deformation to the normal cone and the bundle over it, we follow an algebraic method introduced in \cite{sadegh2018euler} and \cite{Higson2019SpinorsAT}.

In Section \ref{s:equivariantrescaledbundles}, we give an example of the rescaled bundle over the tangent groupoid $\mathbb{T}M$, when $M$ carries a $G$-equivariant Clifford module structure. We give a family of rescaled bundles $\mathbb{E}_J\to \mathbb{T}M$ indexed by $J\in \mathbb{N}$. 

In Section \ref{s:hearkernelasymptoticsproffkirillovformula}, we study the asymptotics of the heat kernel associated with the Bismut Laplacian and its relation to the rescaled bundles, and then we give a proof of Kirillov's character formula for the equivariant index of the Dirac operator.

In Section \ref{examplesofrescaledbundles}, we gave two extra applications of the rescaled bundles in the study of the Witten and Novikov deformations and the equivariant index formula.

%We introduce family of smooth vector bundles over the tangent groupoid that allows us to smoothly deform the heat kernel associated to a Bismut Laplacian to a version of the Mehler's kernel. Then by applying the super-trace we obtain the Kirillov formula.

%The second author wants to thank Yiannis Loizides, Jesus Sanchez and Shiqi Liu for their invaluable suggestions.

%section
\section{The Equivariant Cohomology}\label{s:equivariancohomolgy}

In this section, we recall the Cartan module for equivariant cohomology and the construction of the equivariant characteristic classes. In our exposition, we roughly follow \cite[\S7.1]{berline1992heat}.

\subsection{The polynomial algebra}\label{ss:polynomialalgebra}
Consider the space of polynomials $\bC[\mathfrak{g}]$ that comes with the grading
\[\bC[\mathfrak{g}]=\bigoplus_k\bC_k[\mathfrak{g}]\]
with $\bC^k[\mathfrak{g}]$ consists of homogeneous polynomials of degree $k$.
We also consider the corresponding filtration
\[\bC=\bC^0[\mathfrak{g}]\subset \bC^1[\mathfrak{g}]\subset \cdots \subset \bC[\mathfrak{g}].\]
Denote by $\bC[\mathfrak{g}]_{(J)}$ the quotient of $\bC[\mathfrak{g}]$ by the ideal of polynomials of order $J+1$. Every element $a\in \bC[\mathfrak{g}]_{(J)}$ has a unique polynomial representative of the form $\sum_{\alpha}a_{\alpha}X^{\alpha}$ and hence the algebra $\bC[\mathfrak{g}]_{(J)}$ inherits a grading (by the degree of the polynomials).

The polynomial algebra appears in the Kirillov formula context as 
\[\bC[\mathfrak{g}]\otimes \Lambda V\]
where $V$ is a finite-dimensional vector space. This tensor space has a grading given by 
\begin{equation}\label{equivariantgradingofexterioralgebra}
    \deg(P\otimes a)=2\deg(P)+\deg_{\Lambda}(a)
\end{equation}
where $P\in \bC[\mathfrak{g}]$ and $a\in\Lambda V$. Similar grading will descend to tensor spaces
\[\bC[\mathfrak{g}]_{(J)}\otimes \Lambda V.\]

\subsection{G-equivariant differential forms.}

Assume $G$ is a Lie group acting on the smooth manifold $M$. 
For $X\in \mathfrak{g}$, we denote by $X^M$ or just $X$ the vector field on $M$ corresponding to the infinitesimal action on $C^{\infty}(M)$
\[Xf(m):=\frac{d}{dt}|_{t=0}f(\exp(-tX)m).\]

Let $\bC[\mathfrak{g}]$ be the algebra of polynomial functions on $\mathfrak{g}$ and consider the algebra
\[\bC[g]\otimes\mathcal{A}(M)\]
which carries a left action of $G$: For $\alpha\in \bC[g]\otimes\mathcal{A}(M)$ and $g\in G$ we have
\[(g.\alpha)(X)=g.\Big(\alpha(\textup{Ad}_{g^{-1}}X)\Big).\]
This algebra has a grading for $\alpha\in \bC[g]\otimes\mathcal{A}(M)$ and $g\in G$ given by the equivariant degree
\[o(\alpha)=o^{\textup{ext}}(\alpha)+2o^{\textup{poly}}(\alpha)\]
where $o^{\textup{ext}}$ is the exterior degree of differential forms and $o^{\textup{poly}}$ is the polynomial degree. 

We define the equivariant differential $d_{\mathfrak{g}}$ on $\bC[g]\otimes\mathcal{A}(M)$ 
given by the formula
\[(d_{\mathfrak{g}}\alpha)(X)=d(\alpha(X))-\iota(X^M)\alpha(X).\]
Note that the contraction $\iota(X^M)$ reduces the differential degree by $1$, increases the polynomial degree by $1$ and hence it increases the equivariant degree by  $1.$ Hence $d_{\mathfrak{g}}$ increases the equivariant degree by $1.$ 
Note that $d^2_{\mathfrak{g}}\neq0$, indeed
\[(d^2_{\mathfrak{g}}\alpha)(X)=-\mathscr{L}(X)\alpha(X).\]
However, the square of the equivariant differential vanishes on the {\em Cartan subalgebra}
\[\mathcal{A}_G(M)=\Big(\bC[\mathfrak{g}]\otimes\mathcal{A}(M)\Big)^G\]
 of  invariant elements or equivariant differential forms. Indeed
\begin{itemize}
    \item For every $g\in G$ $$(g.\alpha)(X)=\alpha(\textup{Ad}_{g}X)$$ 
    or equivalently
    \item for every $X\in\mathfrak{g}$ $$\mathscr{L}(X)\alpha(X)=0.$$
\end{itemize}

Since $(\mathcal{A}_G(M),d_{\mathfrak{g}})$ is a complex, we can define its cohomology, the {\em equivariant cohomology} of $M$, denoted by
\[H_G(M).\]
This is known as the {\em Cartan model of equivariant cohomology} \cite{cartan1951notion}. When $M$ is compact and oriented, we define the integration map
\[\int_M:\bC[\mathfrak{g}]\otimes\mathcal{A}(M)\to\bC[\mathfrak{g]}\]
by integrating only terms with the top differential degree part. This map restricts to
\[\int_M:\mathcal{A}_G(M)\to\bC[\mathfrak{g]}^G\]
where the image is the algebra of invariant polynomials. This map indeed descends to the cohomology level:

\[\int_M:H_G(M)\to\bC[\mathfrak{g]}^G.\]

\subsection{The Equivariant Structure}
Let $M$ be compact oriented Riemannian manifold of $M$. Assume a Lie group $G$ acts by positively oriented isometries on $M$. Assume $\mathscr{E}\to M$ is a $G$-equivariant 
%Clifford module with invarinat Hermitian product and Clifford connection 
vector bundle with a connection
$\nabla^{\mathscr{E}}$ that commutes with the $G$-action. We define the moment $\mu^{\mathscr{E}}\in\mathcal{A}_G^2(M,\textup{End}(\mathscr{E}))$ of the connection $\nabla^{\mathscr{E}}$
 by the formula 
 \[\mu^{\mathscr{E}}(X):=\mathscr{L}^{\mathscr{E}}(X)-\nabla_X^{\mathscr{E}}\]
where $\mathscr{L}^{\mathscr{E}}(X)$ is the Lie derivative of $X\in\mathfrak{g}$. In particular, when $\mathscr{E}=TM$ with the Levi-Civita connection $\nabla$, we obtain the Riemannian moment 
%We denote by $R$ the Riemannian curvature of $M$ %and denote by $K$ the curvature of the Clifford bundle.
%Denote the \emph{Riemannian moment map} by
\[\mu^M\in \Big(\Gamma(M,\mathfrak{so}(M))\otimes\mathfrak{g}^*\Big)^G\]
so that for $X\in\mathfrak{g}$ and $\xi\in \Gamma(TM)$ 
\begin{align*}
    \mu^M(X)\xi&=[X,\xi]-\nabla_X\xi\\
    &=-\nabla_{\xi}X.
\end{align*}
Another example of a moment map is the Kosmann formula  for the spinor bundle $\mathscr{E}=S$, when $M$ is spin:

\begin{proposition}\label{p:momentofspinorbundle}\cite{Kosmann1971DrivesDL}
If $M$ is spin, and $\mathscr{E}=S$ is the spinor bundle, then 
\[\mu^S(X)=-\frac{1}{4}c(d\theta_X).\]
%%then the moment  $\mu^S\in \mathcal{A}_G^2(M,\textup{End}({S}))=\mathcal{A}_G^2(M,\bC\textup{l}(TM))$ is given
%\[\mu^S(X)=-\frac{1}{4}c(d\theta_X)\]
%where $\theta_X$ is the one-form dual to $X^M\in\mathfrak{X}(M)$:
%\[\theta_X(\xi)=(X^M,\xi).\]
\end{proposition}
%\begin{proof}
%**********************
%\end{proof}
%By Lemma \ref{momentofspinorbundle}, we define the twisting moment 
%$\mu^{\mathscr{E}/S}$ as the 
Assume $\mathscr{E}\to M$ is a G-equivariant Clifford module whose Clifford action and Hermitian metric are compatible with the $G$-action. Locally, such vector bundle is given as $S\otimes E$ where $S$ is the local spinor bundle and $E$ is an equivariant auxiliary vector bundle with $G$-compatible connection $\nabla^E$ such that $\nabla^{\mathscr{E}}=\nabla^S\otimes 1+1\otimes\nabla^{E}.$  Then we define the twisting moment map (locally) as the moment of the auxiliary bundle $E$ or, equivalently, the difference \[\mu^{\mathscr{E}/S}:=\mu^{\mathscr{E}}-\mu^{S}.\]
%the difference of the moments of the bundles $\mathscr{E}$ and $S$
So we have $\mu^{\mathscr{E}/S}=\mu^{\mathscr{E}}+\frac{1}{4}c(d\theta_X)$, and moreover, we have $\mu^{\mathscr{E}/S}\in\mathcal{A}_G^2(M,\textup{End}_{\textup{Cl}}(\mathscr{E}))$ (see \cite{berline1992heat}). Here $\textup{End}_{\textup{Cl}}$ denotes the space of endomorphisms that commute with the Clifford action.

The equivariant curvature of $M$ is defined by 
\[R_{\mathfrak{g}}:=R+\mu^M\in \mathcal{A}_G^2(M,\mathfrak{so}(M)).\]
where $R$ is the Riemannian curvature tensor. The equivariant $\Ahat$-genus is the closed equivariant differential form
\[\Ahat_{\mathfrak{g}}(X,M):=\textup{det}^{1/2}\Big(\frac{R_{\mathfrak{g}}/2}{\sinh(R_{\mathfrak{g}}/2)}\Big)\in\mathcal{A}_G(M).\]
Using this, we obtain the equivariant twisting curvature of $\mathscr{E}$ given by
\[F^{\mathscr{E}/S}_{\mathfrak{g}}(X)=F^{\mathscr{E}/S}+\mu^{\mathscr{E}/S}\in\mathcal{A}_G^2(M,\textup{End}_{\textup{Cl}}(\mathscr{E})).\]
The equivariant relative Chern character of $\mathscr{E}$ is then defined as 
\[\Ch_{\mathfrak{g}}(X,\mathscr{E}/S)=\Str_{\mathscr{E}/S}\Big(\exp(-F^{\mathscr{E}/S}_{\mathfrak{g}}(X))\Big)\in \mathcal{A}_G^2(M).\]

%section
\section{The Equivarinat Index and the Kirillov Formula}\label{s:equivariantindexandkirillovformula}

In this section, we formulate our main result – the Kirillov formula for the index of the equivariant Dirac-type operator. Then we give an overview of the heat kernel proof of this formula, \cite[Ch.~8]{berline1992heat}. In particular, we define the Bismut Laplacian \cite{bismut1985index} and explain its role in the proof. 

\subsection{The settings}\label{ss:settings}
Let $G$ be a compact Lie Group acting by positively oriented isometries on a compact oriented even-dimensional Riemannian manifold $M^n$. Consider a $\bZ_2$-graded Clifford module  $\mathscr{E}\to M$, which carries an action of $G$ that commutes with the grading, the Hermitian product, and the Clifford action. We fix a Clifford a $G$-equivariant connection $\nabla^{\mathscr{E}}$ on $\mathscr{E}$ and consider the Dirac-type operator 
\[D:\Gamma(M,\mathscr{E})\xrightarrow[]{\nabla^{\mathscr{E}}}\Gamma(M,\mathscr{E}\otimes\Lambda^1T^*M)\xrightarrow[]{c}\Gamma(M,\mathscr{E}).\]
Then $D$ commutes with the action of $G$. Any $g\in G$, takes $\ker(D^{\pm})$ to themselves. The {\em equivariant index} $\textup{ind}(g,D)$ of $D$ is defined as the supertrace
\[\textup{ind}(g,D):= \textup{tr}(g,\ker(D^+))-\textup{tr}(g,\ker(D^-)).\]

\begin{theorem}[Kirillov Formula]\label{th:kirillovformula} For $g=e^{-X}$, with $X\in\mathfrak{g}$ sufficiently small, 
\[\textup{ind}(e^{-X},D)=(2\pi i)^{-n/2}\int_M\Ahat_{\mathfrak{g}}(X,M)\Ch_{\mathfrak{g}}(X,\mathscr{E}/S)\]
\end{theorem}
In the rest of this section, we give a brief overview of the heat kernel proof of this theorem. 

\subsection{The modified McKean-Singer formula}\label{ss:McKeanSinger}
The equivariant index of $D$ can be related to the heat kernel of $D$ via the  McKean-Singer-type theorem:
\begin{equation}\label{eq:mckeansinger}
    \textup{ind}(g,D)=\textup{str}(ge^{-tD^2})
\end{equation}
for $t>0.$ 
However, it turns out that the asymptotic equation of the heat kernel 
in the right hand side of this formula is complicated, and the Getzler's rescaling method does not go through well for it. The calculation becomes easier if we  slightly modify the operator $D$. In fact, the following generalization of \eqref{eq:mckeansinger} is true: Fix $X\in \mathfrak{g}$ and for $u\in\bC$ define the deformed Dirac operator
\[D_u=D+uc(X).\]
Note that this is the Dirac operator associated with the connection 

\begin{equation}\label{eq:deformedconnection}
    \nabla^{\mathscr{E},-uX}:=\nabla^{\mathscr{E}}+u\theta_X.
\end{equation}
where $\theta_X\in \mathcal{A}^1(M)$ is the one-form dual to the vector field $X^M$:
\[\theta_X(\xi)=(X^M,\xi),\]
in other words
\[D_u:\Gamma(M,\mathscr{E})\xrightarrow[]{\nabla^{\mathscr{E},-uX}}\Gamma(M,\mathscr{E}\otimes\Lambda^1T^*M)\xrightarrow[]{c}\Gamma(M,\mathscr{E}).\]
\begin{proposition}\label{mckeansingergeneralization} For every $t>0$ and $u\in\bC$ we have
\[\textup{ind}(e^{-X},D)=\textup{str}(e^{-X}e^{-tD^2_u}).\]
\end{proposition}
 Note that since $D$ commutes with the infinitesimal action of $\mathfrak{g}$,  we can write
\[e^{-X}e^{-tD^2_u}=e^{-\mathscr{L}(X)}e^{-tD^2_u}=e^{-tD_u^2-\mathscr{L}(X)}.\]
We denote the generalized Laplacian operator that appears in the exponent of the right-hand side, for $t=1$,  by

\begin{equation}\label{eq:bismutlaplacian}
    H_u(X)=D_u^2+\mathscr{L}^{\mathscr{E}}(X).
\end{equation}
Therefore the equivariant index corresponding to $e^{-X}\in G$ is given by 
\[\textup{ind}(e^{-X},D)=\textup{str}(e^{-H_u(X)}).\]

\subsection{The Lichnerowicz formula}\label{ss:Lichnnerowicz}
We now compute the operator $H_u$.

\begin{lemma}\cite[Prop.~3.45]{berline1992heat}\label{commutatorformuladiracclifford}
If $\theta\in\mathcal{A}^1(M)$ is dual to the vector field $X\in\mathfrak{X}(M)$ then
\[[D,c(d\theta)]=-2\nabla_X^{\mathscr{E}}+c(d\theta)+d^*\theta.\]

\end{lemma}

\begin{proposition}\label{p:lichnerowicztypeformula} We have a Lichnerowicz-type formula for $H_u(X)$:
\[H_u(X)=\Big(\nabla^{\mathscr{E},uX}\Big)^*\nabla^{\mathscr{E},uX}+\frac{1}{4}{\kappa}_M+c(F^{\mathscr{E}/S})+\mu^{\mathscr{E}}(X)+u c(d\theta_X)+(1-4u)\nabla_X^{\mathscr{E}}.\]
With respect to a local orthonormal frame $\{e_i\}$ 
\[\Big(\nabla^{\mathscr{E},uX}\Big)^*\nabla^{\mathscr{E},uX}=-\sum_i\Big(\nabla^{\mathscr{E},uX}_{e_i}\Big)^2-\nabla^{\mathscr{E},uX}_{\nabla_{e_i}e_i}.\]
Note that when $u=\frac{1}{4}$, using the equality $\mu^{\mathscr{E}/S}=\mu^{\mathscr{E}}+\frac{1}{4}c(d\theta_X)$, we obtain
\[H_{1/4}(X)=\Big(\nabla^{\mathscr{E},\frac{1}4X}\Big)^*\nabla^{\mathscr{E},\frac{1}4X}+\frac{1}{4}{\kappa}_M+c(F^{\mathscr{E}/S})+\mu^{\mathscr{E}/S}(X).\]

%where $\Delta_{uX}$ is the Laplacian associated to the connection $\nabla^{\mathscr{E},uX}=\nabla^{\mathscr{E}}-u\theta_X$, namely 
\end{proposition}

\begin{proof}
 Note that since the $D_u$ is the Dirac operator associated to the connection $\nabla^{\mathscr{E},-uX}$, and $\nabla^{\mathscr{E},-uX}=\nabla^{\mathscr{E}}+u\theta_X$ is a Clifford connection; then we have the Lichnerowicz formula
 
 \begin{align*}
     D_u^2&=\Big(\nabla^{\mathscr{E},-uX}\Big)^*\nabla^{\mathscr{E},-uX}+\frac{1}{4}\kappa_M+c(F^{\mathscr{E}/S,-u})\\
     &=\Big(\nabla^{\mathscr{E},-uX}\Big)^*\nabla^{\mathscr{E},-uX}+\frac{1}{4}\kappa_M+c(F^{\mathscr{E}/S})+u c(d\theta_X)
 \end{align*}
 We may write a similar formula for $D_{-u}^2$, and by subtracting, we have
 \begin{align*}
     D_u^2-D_{-u}^2&=\Big(\nabla^{\mathscr{E},-uX}\Big)^*\nabla^{\mathscr{E},-uX}-\Big(\nabla^{\mathscr{E},uX}\Big)^*\nabla^{\mathscr{E},uX}+2uc(d\theta_X)
 \end{align*}
 we also have
 \begin{align*}
     D_u^2-D_{-u}^2&=2u[D,c(\theta_X)]\\
     &=-4u\nabla^{\mathscr{E}}_X+2uc(d\theta_X)+2ud^*\theta_X
 \end{align*}
 where we used Lemma \ref{commutatorformuladiracclifford} in the second equality. Note that $d^*\theta_X=\tr(\mu^M(X))=0$. Combining these formulas, we find
 \begin{align*}
     H_u(X)&=D_u^2+\mathscr{L}^{\mathscr{E}}(X)\\
     &=\Big(\nabla^{\mathscr{E},uX}\Big)^*\nabla^{\mathscr{E},uX}+\frac{1}{4}\kappa_M++c(F^{\mathscr{E}/S})+uc(d\theta_X)-4u\nabla_X^{\mathscr{E}}+\mathscr{L}^{\mathscr{E}}(X).
 \end{align*}
 Using the equality $\mathscr{L}^{\mathscr{E}}(X)=\nabla^{\mathscr{E}}_X+\mu^{\mathscr{E}}(X)$ we then achieve the result.
 
\end{proof}

\subsection{The Bismut Laplacian}\label{ss:Bismut}
As we see, for $u=1/4$, we obtain a simpler Lichnerowicz formula. The operator
\begin{equation}\label{bismutlaplacianu=1/4}
    H_{1/4}=D^2_{1/4}+\mathscr{L}^{\mathscr{E}}(X)
\end{equation}
is called the {\em Bismut Laplacian}.  From now on, we  drop the $1/4$ from the Bismut Laplacian notation and denote it by 
$H(X):=H_{1/4}(X)$. Following \cite{berline1992heat} we replace $D^2$ with $H_{1/4}$ in \eqref{eq:mckeansinger}. 

\subsection{The heat kernel in normal coordinates}\label{ss:innormalcoordinates}
The proof of the Kirillov formula \ref{th:kirillovformula}, is done through rescaling of the smoothing kernel of the (a conjugate of the) heat operator $e^{-tH(X)}$ in normal coordinates. Roughly one needs to trivialize the bundle $\mathscr{E}\to M$ over a small neighborhood using the normal coordinates $\br{x}\to \exp_m(\br{x})$, through parallel transport along radial geodesics. Here $\br{x}$ belongs to a small neighborhood of the origin in $T_mM$.

The Bismut Laplacian $H(X)$ is associated to the perturbed connection $\nabla^{\mathscr{E},\frac{1}{4}X}=\nabla^{\mathscr{E}}-\frac{1}{4}\theta_X$. So it is natural to use the parallel transport of this connection to define the trivialization of $\mathscr{E}$. However, it is more convenient to do a ``two-step" trivialization: first, trivialize the bundle using the connection $\nabla^{\mathscr{E}}$ and compute the heat kernel $e^{-tH(X)}$ in this trivialization  and then consider the conjugated kernel   $\rho_{\frac{1}4X}(\br{x})e^{-tH(X)}\rho_{\frac{1}4X}(\br{x})^{-1}$, where $\rho_{\frac{1}4X}(\br{x})$ is the parallel transport map on the trivial line bundle along the radial geodesics with respect to the connection $d-\frac{1}4\theta_X$ (see \eqref{eq:rhomap}). In this approach, in particular, the dependence of the result on $X$ is more explicit.

%\section{de Rham Equivariant Cohomology}
\section{Rescaled Bundles}\label{s:rescaledbundles}

In this paper, we achieve the rescaling calculus through a coordinate-free method, i.e. through tangent groupoid of the manifold $M$. It is, however, more general and arguably more natural to consider a more general deformation space $\DNC(V,M)$.
%So in this and the next sections we describe the construction of tangent groupoid and introduce a vector bundle on it such that it recovers the Getzler's symbol calculus (cf. \cite{Higson2019SpinorsAT}).
So in this and in the next sections, we introduce a setup that gives vector bundles over the deformation space $\DNC(V,M)$, generalizing the rescaled bundle of \cite{Higson2019SpinorsAT}. Specifying to the case of the tangent groupoid, we recover the  Getzler symbol calculus.

\subsection{Deformation to the normal cone}\label{ss:deformationtothenormalcone}
Assume $i:M\hookrightarrow V$ is an embedding of smooth manifolds. The normal bundle of this embedding is the quotient bundle $\mathcal{N}=TV|_M/TM$.  
As a set the deformation space $\textup{DNC}(V,M)$ is given by the disjoint union 
\[\textup{DNC}(V,M):=\mathcal{N}\times\{0\}\bigsqcup_{t{\neq 0}} V\times \{t\}.\]
To obtain a smooth structure, fix a Riemannian metric $g$ on $V$ and identify $\mathcal{N}$ with $TM^{\perp}$ the orthogonal complement of $TM$ in $TV|_M$.
Then we have $\textup{Exp}^g:\mathcal{W}\to U$, the Riemannian exponential map, from a neighborhood map $\mathcal{W}$ of $0$-section of $\mathcal{N}$ to a neighborhood $U$ of $M$ in $V$. The smooth structure on $\textup{DNC}(V,M)$ away from the normal bundle is given by the product manifold $M\times\mathbb{R}^{\neq 0}$. To obtain the manifold structure in a neighborhood of $\mathcal{N}$ we require the following map to be a diffeomorphism (c.f. \cite{Debord2017BlowupCF}):
\[\mathcal{W}\times\mathbb{R}\to \textup{DNC}(V,M)\]
\[(x,X,t) \mapsto\left\{
	\begin{array}{ll}
	 (\textup{Exp}_x^g(tX),t)   & t\neq0 \\
	
	(x,X,0) & t=0
	\end{array}
\right.\]
One may obtain local coordinates for this deformation space explicitly; see \cite{rouse2008schwartz, sadegh2018euler}.

For the construction of the rescaled bundle, we will use an algebraic approach to the deformation to the normal cone, following \cite{sadegh2018euler}.
This is inspired by the algebraic geometric definition of the deformation to the normal cone as the prime spectrum of the \emph{Rees algebra} (see for example \cite[Chapter ~5]{fulton2013intersection} and  \cite{higson2010tangent}).

Let $I_M\subset C^{\infty}(V)$ be the vanishing ideal of $M$. Consider the Rees algebra $\mathcal{A}(V,M)\subset C^{\infty}(V)[t,t^{-1}]$ given by
\[\mathcal{A}(V,M)=\bigoplus_{p=-\infty}^\infty I_M^pt^{-p}.\]
Therefore the Rees algebra consists of the Laurent polynomials
\[\sum_p f_pt^{-p}\]
where $f_p$ vanishes to $p$-th order along $M$. 

\begin{definition}\label{d:characterspectrum}
For a $\mathbb{C}$-algebra $\mathcal{A}$, the {\em character spectrum}, $\textup{Spec}(\mathcal{A})$, is the space of all algebra homomorphisms
\[\phi:\mathcal{A}\to\mathbb{C}\]
with the weak topology of pointwise convergence.
\end{definition}

%The spectrum of $\mathcal{A}(V,M)$, denoted by $\textup{Spec}(\mathcal{A}(V,M))$ is the space of all algebra homomorphism
%\[\phi:\mathcal{A}(V,M)\to \mathbb{C}.\]
%The spectrum carries a natural topology of pointwise convergence. 
It turns out the spectrum of the algebra $\mathcal{A}(V,M)$ consists of the following characters:
\begin{enumerate}
    \item For every $(v,\lambda)\in V\times\mathbb{R}^{\neq0} $
    \[\varepsilon_{(v,\lambda)}:\mathcal{A}(V,M)\to \mathbb{C}\]
\begin{equation}\label{genericcharacters}
    \sum_pf_pt^{-p}\mapsto\sum_pf_p(v)\lambda^{-p}
\end{equation}
    
    \item For $X_m\in\mathcal{N}_m=T_mV/T_mM$
    \[\varepsilon_{X_m}:\mathcal{A}(V,M)\to \mathbb{C}\]
    \begin{equation}\label{nongenericcharacters}
        \sum_pf_pt^{-p}\mapsto\sum_p\frac{1}{p!}X^pf_p(m)
    \end{equation}%\mn{what is $\lambda$?}
     where $X\in\mathfrak{X}(V)$ is a vector field that represents $X_m\in \mathcal{N}_m$ at $m\in M.$
\end{enumerate}
Hence there is a one-to-one correspondence
\[\textup{Spec}(\mathcal{A}(V,M))\longleftrightarrow \mathcal{N}\times\{0\}\bigsqcup_{t\neq 0}V\times\{t\}.\]
This spectrum is indeed a manifold:
\begin{theorem}\cite{sadegh2018euler}
The spectrum $\textup{Spec}(\mathcal{A}(V,M))$ is a smooth manifold of dimension $\dim(V)+1$ that has a canonical submersion to $\mathbb{R}$. This manifold is called the {\em deformation to the normal cone} and is denoted by $\DNC(V,M)$. As a set, it is  given by a disjoint union of fibers over $\mathbb{R}$ as follows
\[\mathcal{N}\times\{0\}\bigsqcup_{t\neq 0}V\times\{t\}.\]
\end{theorem}
To gain insight into this theorem, it is useful to see  the algebra $\mathcal{A}(V,M)$ for $\DNC(V,M)$ as a ``homogeneous coordinate ring'' for a variety. Note that the vanishing order along the submanifold $M$, gives a filtration 
\[C^{\infty}(V)=I_0\supset I_1\supset\cdots\]
where $I_p=I_M^p$ is the ideal of functions vanishing to order $p$ along $M$.
Denote by $\mathcal{A}_0(V,M)$ the quotient algebra $\mathcal{A}(V,M)/t\mathcal{A}(V,M)$ which is naturally isomorphic to the associated graded algebra:
\[\mathcal{A}_0(V,M)\buildrel{\simeq}\over\longrightarrow\bigoplus_{p=0} I_p/I_{p+1}\]
\[\sum_pf_pt^{-p}\mapsto\sum_{q\geq0}\langle f_p\rangle_p\]
We have the inclusion $C^{\infty}(M)\hookrightarrow\mathcal{A}_0(V,M)$ as the zero-degree part. This corresponds to
the following fact 
\begin{lemma}
     The character spectrum of $\mathcal{A}_0(V,M)$ is the normal bundle $\mathcal{N}\to M.$ The quotient $\mathcal{A}_{0,m}(V,M)$  of the algebra $\mathcal{A}_0(V,M)$ by the ideal of functions vanishing at $m\in M$ is isomorphic to the polynomial algebra over the normal space $\mathcal{N}_m$:
     \[\sum_pf_pt^{-p}\mapsto [X_m\mapsto \sum_p\frac{1}{p!}X^p.f_p|_m]\]
     for $X_m\in\mathcal{N}_m.$
\end{lemma}

 One important example of the deformation to the normal cone is the \emph{tangent groupoid}:
\begin{example}\label{tangentgroupoid}
Let $M$ be a smooth manifold. The deformation to the normal cone associated to the diagonal embedding $M\hookrightarrow M\times M$ is called the tangent groupoid and is denoted by
\[\mathbb{T}M:=\DNC(M\times M,M).\]
As a set, it is given by the disjoint union
\[M\times M\times\bR^{\neq 0}\bigsqcup TM\times\{0\}\]
which is naturally fibered over $\bR$.

The tangent groupoid is indeed a smooth groupoid, $\bT M\rightrightarrows{M\times\bR}$. 
 The source and target maps are given so that for $t\neq 0$, they restrict to pair groupoid
\[M\times M\times\{t\}\rightrightarrows{M}\times\{t\}\]
and for $t=0$, the source and target map are the same as the projection map of the vector bundle:
\[TM\times\{0\}\to M\times\{0\}.\]

\end{example}

%whose spectrum is the normal bundle $\mathcal{N}\to M$ corresponding to the  
\subsection{Alternative  formula for $\epsilon_{X_m}$}\label{ss:alternativeXm}
We now describe Higson and Yi's alternative presentation of the formula \eqref{nongenericcharacters}, \cite{Higson2019SpinorsAT}, which justifies many formulas in the upcoming subsections and sections.
For every vector field $X\in\mathfrak{X}(V)$ we have the derivation map
\[tX:\mathcal{A}(V,M)\to \mathcal{A}(V,M)\]
\[\sum_pf_pt^{-p}\mapsto \sum_pX.f_pt^{-p+1}\]
which descends to the quotient algebra
\[tX:\mathcal{A}_0(V,M)\to \mathcal{A}_0(V,M).\]
The action of $tX$ on $\mathcal{A}_0(V,M)$ is  locally nilpotent, in the sense that for each $f\in \mathcal{A}_0(V,M)$ there exists an integer $n=n(f)$ such that $(tX)^nf=0$. Hence, we may define the exponential homomorphism
\[\exp(tX):\mathcal{A}_0(V,M)\to \mathcal{A}_0(V,M)\]
as the {\em finite sum}:
\[\exp(tX)f:= \sum_j \frac{(tX)^jf}{j!}, \qquad f\in \mathcal{A}_0(V,M).\]

Now we may rewrite the formula \eqref{nongenericcharacters} as the composition
\[\varepsilon_{X_m}=\varepsilon_m\circ\exp(tX)\]
where 
\[\varepsilon_m=\varepsilon_{0_m}:\mathcal{A}_0(V,M)\to\mathbb{C},\qquad \sum_pf_pt^{-p}\mapsto f_0(m)
\]
is the evaluation map at $X_m=0$.

\subsection{The scaling order}\label{ss:Scalingorder}

%One way to introduce a vector bundle over $\textup{DNC}(V,M)$ is as follows.

Let $E\to V$ be a vector bundle such that the restricted vector bundle $F=E|_M\to M$ has a filtration $F^1\subset F^2\subset \cdots \subset F^q=F$.
We assume the bundle of endomorphisms $\textup{End}(E)$ has an algebra filtration 
\begin{equation}\label{eq:FiltrationEnd}
\textup{End}(E)= \textup{End}(E)^0\subset \textup{End}(E)^1\subset\cdots\subset \textup{End}(E)^q= \textup{End}(E),
\end{equation}
whose restriction to $M$ is compatible with the filtration on $F$ in the sense that 
\begin{equation}\label{eq:comatibilityEnd}
   A:F^\bullet \to F^{\bullet+j}, \qquad\text{for}\quad A\in \textup{End}(E)^j.
\end{equation}

We consider a connection $\nabla$ on $E\to V$
with the following properties:
\begin{itemize}
    \item The induced connection $i^*\nabla$ is compatible with the filtration in the sense that 
    %for every $1\leq j\leq q$
\begin{equation}\label{technicalcondotionforrescaledbundle1}
    i^*\nabla:\Gamma(F^j)\to \Gamma(F^j)\otimes\Omega^1(M), 
    \qquad j=1\ldots q.
\end{equation}
  
  \item The curvature $K$ of $\nabla$ has filtration order at most 2:
  \begin{equation}\label{technicalcondotionforrescaledbundle2}
      K\in \Gamma(\textup{End}(E)^2)\otimes \Omega^2(M).
  \end{equation}
  %For every $X_m,Y_m\in TV_m$, for $m\in M$, the curvature operator $K(X_m,Y_m)$ of the connection $\nabla$ is of order at most 2:
  
      \item The induced connection $\nabla^{\textup{End}(E)}$  {has filtration order 0}:
    \begin{equation}\label{technicalcondotionforrescaledbundle3}
    \nabla^{\textup{End}(E)}:\Gamma(\textup{End}(E)^j)\to \Gamma(\textup{End}(E)^{j})\otimes\Omega^1(M).
\end{equation}
    
    %For the curvature tensor $F$ of the connection $\nabla$, for every two vectors $X_m,Y_m\in TV_m$, with $m\in M$,  the endomorphism $F(X_m,Y_m)$ has filtration order at most $2$:

\end{itemize}

%Note that the bundle endomorphisms, $\textup{End}(F)$, inherit a filtration and through restriction the bundle endomorphisms of $\textup{End}(E)$ has a induced filtration. 
The space of sections $\Gamma(\textup{End}(E))$ is filtered by the algebra filtration of $\textup{End}(E)$. We use the notation $o^g(\phi)$ for the order of $\phi\in\Gamma(\textup{End}(E))$ in this filtration. By definition we have
\begin{equation}\label{subadditiveproperty}
    o^g(\phi_1\phi_2)\leq o^g(\phi_1)+o^g(\phi_2).
\end{equation}
The filtration $F^\bullet$ induces a filtration on $\textup{End}(E|_M)$, which might differ from the filtration induced by \eqref{eq:FiltrationEnd}. %which might be different from \eqref{eq:FiltrationEnd}.
For a section $\phi\in \Gamma(\textup{End}(E))$ we denote by $o^f(\phi)$ the order of this filtration of the restriction $
\phi|_M$. Then it follows from \eqref{eq:comatibilityEnd} that 
\[o^f(\phi)\leq o^g(\phi).\]
Note, however, that the order function $o^f$ fails to satisfy the property \eqref{subadditiveproperty}, in general. %{So we use $o^g$ to define a filtration on the space of differential operators}:

%So for every nonzero section $\phi\in\Gamma(\textup{End}(E))$ we have an order $o^{f}(\phi)$. Using these we obtain a new filtration on the differential operators on the bundle $E\to V$:
\begin{definition}\label{def:Getzlerorder}
A differential operator $D$ acting on $\Gamma(E)$ is of \emph{Getzler order} at most $p$ if locally it can be written as the sums of terms of the form
\[\phi\nabla_{X_1}\cdots\nabla_{X_l}\]
where $\phi\in \Gamma(\textup{End}(E))$ and $X_1,\cdots,X_l\in\mathfrak{X}(M)$ where $o^g(\phi)+l\leq p.$ We denote the Getzler order of $D$ by $o^g(D).$
\end{definition}

We had to use the filtration order $o^g$, rather than  $o^f$,  on $\Gamma(\textup{End}(E))$  in the definition of Getzler order to obtain the following:

\begin{lemma}
     For two differential operator $D_1, D_2$ acting on $\Gamma(E)$ we have
     \[o^g(D_1D_2)\leq o^g(D_1)+o^g(D_2).\] \qed
\end{lemma}

The space of sections $\Gamma(E)$ has natural filtration obtained by restricting to the submanifold $M$. We denote the corresponding filtration order by $o^f(\sigma)$ for $\sigma\in \Gamma(E)$ (if $\sigma=0$ we set $o^f(\sigma)=-\infty$).

We need yet another filtration on $\Gamma(E)$, defined by the \emph{scaling order} 
\[o^{sc}(\sigma)=\min_{D}\{o^{g}(D)-o^{f}(D\sigma)\}.\]

\begin{remark}
The definition of the scaling order might seem unnatural at first. But it is modeled on the following formula for the vanishing order of a function:

For $f\in C^{\infty}(V)$, define $o^{{val}}(f)$ to be $0$ if $f|_{M}$ is not the zero function, otherwise define
$o^{val}(f)=-\infty.$ One easily sees that the vanishing order of $f$ along the submanifold $M$, $o^{van}(f)$, is given by the equality
\[o^{van}(f)=\min_D\{o(D)-o^{val}(Df)\}\]
where $o(D)$ is the ordinary differential order  and the minimum is taken over all differential operators on $V$.
\end{remark}
 It follows from the definition of the scaling order that
\begin{lemma}
     For a differential operator $D$ acting on $\Gamma(E)$ with $o^g(D)\leq q$ and $\sigma\in \Gamma(E)$ we have
     \[o^{sc}(D\sigma)\geq o^{sc}(\sigma)-q.\]
     \qed
\end{lemma}

\subsection{The Taylor order} In this subsection, we obtain a formula for the scaling order in local coordinates. 

\begin{definition}\label{def:Eulerlike}
Let $M$ be an embedded submanifold of a manifold $V$. A vector field $\mathcal{R}\in\mathfrak{X}(V)$ is called Euler-like for the embedding $M\to V$, if for every function $f\in C^{\infty}(V)$ vanishing to  $p$-th order along $M$ 
\[\mathcal{R}f=pf+g\]
where $g\in C^{\infty}(V)$ vanishes to order $p+1$ along $M$.
\end{definition}

The Euler-like vector fields can be used to obtain a tubular neighborhood embeddings:

\begin{theorem}\label{tubularnghbdequivalenteulerlikevectorfield}\cite{bursztyn2019splitting, sadegh2018euler}
There is a bijection between the germs of Euler-like vector fields and germs of tubular neighborhood embeddings.
\end{theorem}

The correspondence given by this theorem can be described as follows: Given a tubular neighborhood embedding we can choose local coordinates 
\begin{equation}\label{trivializingneighborhood}
    (x^1,\cdots,x^l,y^1,\cdots,y^k):\mathcal{N}\to \mathbb{R}^l\times \mathbb{R}^k
\end{equation}
near a point of $M$, where $x_i$'s form local coordinates on $M$ and $y_j$'s are the linear coordinates for the fibers of the normal bundle. Under the correspondence in Theorem \ref{tubularnghbdequivalenteulerlikevectorfield}, the associated Euler-like vector field in the trivializing neighborhood is given by
\[\mathcal{R}=\sum_jy^j\partial_{y^j}.\]

Using the trivialization \eqref{trivializingneighborhood}, we call a vector field \emph{horizontal} if it is a linear combination of coordinate vector fields $\partial_{x^i}$'s and \emph{vertical} if it is a linear combination of coordinate vector fields $\partial_{y^j}$'s. In particular, the Euler-like vector field is a vertical vector field.

%Fix a trivialization \eqref{trivializingneighborhood} and its associated Euler-like vector field $\mathcal{R}.$ 
\begin{definition}
A section in $\sigma\in\Gamma(E)$ is called {\em $\mathcal{R}$-synchronous} if 
\[\nabla_{\mathcal{R}}\sigma=0\]
in a neighborhood of $M$.
\end{definition}
Using this concept, we define the Taylor expansion of a section $\sigma\in \Gamma(E)$ in the trivializing neighborhood:
The formal sum $\sum_I \sigma_Iy^I$
is a Taylor expansion for $\sigma$ if 
\begin{itemize}
    \item $\sigma_I\in\Gamma(E)$ is an $\mathcal{R}$-synchronous section.
    \item For every $N\geq0$, the difference
    \[\sigma-\sum_{|I|<N}\sigma_Iy^I\]
    vanishes to $N$-th order.
\end{itemize}
\begin{definition}\label{d:Taylororder}
For a section $\sigma\in\Gamma(E)$ with support in the trivializing neighborhood, with Taylor expansion $\sum_I \sigma_Iy^I$ we define the \emph{Taylor order} by the formula
\[o^t(\sigma):=\min_{|I|}\{|I|-o^{f}(\sigma_{I})\}.\]
\end{definition}
We will show that the Taylor order equals the scaling order. To prove this, we need the following lemma:
\begin{lemma}\label{l:normalderivativesynchronous}
     Let $\sigma\in\Gamma(E)$ be an $\mathcal{R}$-synchronous section. 
     \begin{itemize}
         \item If $Y$ is a vertical vector field, then we have the Taylor expansion
         \[\nabla_Y\sigma\simeq \sum_{|I|>0}\omega_I(\sigma)y^I\]
         where $\omega_I$'s are endomorphisms of the vector bundle $E$ of order at most $2$.
         \item If $X$ is a horizontal vector field, then we have the following Taylor expansion
         \[\nabla_X\sigma\simeq \sigma_0+\sum_{|I|>0}\eta_I(\sigma)y^I\]
         where  $\eta_I$'s are  endomorphisms of the vector bundle $E$ of order at most $2$ and $o^f(\sigma_0)\leq o^f(\sigma)$. 
     \end{itemize}
\end{lemma}

\begin{proof}
For the first bullet point, without loss of generality, we may assume $Y=\partial_i$, where we have $[\mathcal{R},Y]=-Y$. Consider the equation
\[\nabla_{\mathcal{R}}\nabla_Y\sigma-\nabla_Y\nabla_{\mathcal{R}}\sigma -\nabla_{[\mathcal{R},Y]}\sigma=K(\mathcal{R},Y)\sigma.\]
Since $\nabla_{\mathcal{R}}\sigma=0$, we obtain
\begin{equation}\label{pretaylorexpansion}
    \nabla_{\mathcal{R}}\nabla_{Y}\sigma+\nabla_{Y}\sigma=K(\mathcal{R},Y)\sigma.
\end{equation}
Let $\sum_I\sigma_Iy^I$ be the Taylor expansion for $\nabla_{Y}\sigma$ and let 
\[\sum_{|I|>0}\omega_Iy^I\]
be the Taylor expansion of the endomorphism $K(\mathcal{R},Y)$, where $\omega_I$'s are endomorphisms of $E$ where $\nabla_{\mathcal{R}}\omega_I=0$. Now by writing the Taylor expansion of both sides of \eqref{pretaylorexpansion}, we obtain
\[\sum_{I}(1+|I|)\sigma_Iy^I=\sum_{|I|>0}\omega_I(\sigma)y^I\]
from which the statement follows.

The argument for the second bullet point is quite similar, except for the inequality $o^f(\sigma_0)\leq o^f(\sigma)$, where it follows from the compatibility of the connection with filtration of the restricted bundle $E|_M$.
\end{proof}

\begin{corollary}\label{c:generalderivativesynchronous}
If  $\sigma\in \Gamma(E)$ is an $\mathcal{R}$-synchronous section and $D$ is a differential operator  of Getzler order $k$. Then for the Taylor expansion 
\[D\sigma\sim \sum_{I}y^{I}\sigma_{I}\]
we have
\[o^{f}(\sigma_{I})\leq o^{f}(\sigma)+|I|+k.\]
So, in particular,
\[o^{f}(D\sigma)\leq o^{f}(\sigma)+k.\]
\end{corollary}

\begin{proof}
If the $D$ is an endomorphism of the bundle $E\to V$, the statement is obvious.
Consider a vector field $X$ on $V$. We will prove the statement for $D=\nabla_X$, and by induction on the differential order, the full generality follows.  %$D=c(X,Y)$ the case is obvious.

We can write  $X=X^{\textup{h}}+X^{\textup{v}}$, where $X^{\textup{h}}$ and $X^{\textup{v}}$ are the horizontal and vertical components of $X$. Now from Lemma \ref{l:normalderivativesynchronous} the corollary follows.
%, for $D=\nabla_{(X,Y)}$, it follows 
%\[o^t(\nabla_{(X,Y)}\sigma)\geq o^t(\sigma)-1\]
%\[o^{\gamma,cl}(\sigma_{\alpha})\leq o^{\gamma,cl}(\sigma)+|\alpha|+1.\]
%Now by induction the generalized case follows.
\end{proof}

Recall that the Taylor order was defined in Definition~\ref{d:Taylororder}.
\begin{theorem}\label{t:taylororderequalscalingorder}
For a section $\sigma\in\Gamma(E)$ with support in the trivializing neighborhood \eqref{trivializingneighborhood}, we have
\[o^{sc}(\sigma)=o^{t}(\sigma).\]
\end{theorem}
\begin{proof}
Choose $I=(\alpha_1,\cdots,\alpha_k)$ with smallest $|I|$ such that 
\[o^{t}(\sigma)=|\alpha|-o^{f}(\sigma_{I}),\]
and define the differential operator $D=\nabla^{\alpha_1}_{\partial_{y^1}}\cdots\nabla^{\alpha_{n}}_{\partial_{y^n}}$. 
%It follows from Corollary \ref{c:generalderivativesynchronous} that 
We claim that
\[o^{f}(D\sigma)=o^{f}(\sigma_{\alpha}).\]
For any $J\neq I$ with $|J|\geq|I|$ we have
\[D(y^{J}\sigma_{J})|_{M}=0.\]
If $|J|<|I|$, by the choice of $I$, we have
\[|J|-o^{f}(\sigma_{J})>|I|-o^{f}(\sigma_{I}).\]
Also, by corollary \ref{c:generalderivativesynchronous}, we have
\[o^{f}(Dy^{J}\sigma_{J})\leq o^{f}(\sigma_{J})+|I|-|J|\]
and combining this with the previous inequality gives
\[o^{f}(Dy^{J}\sigma_{J})<o^{f}(\sigma_{I})\]
from which our claim follows. Therefore
\[o^{sc}(\sigma)\leq o^{t}(\sigma).\]
Now we prove the opposite inequality. %Consider a vector field $(X,Y)\in\Gamma(TM\times TM)$. We have $(X,Y)=(X,0)+(Y,Y)$ where $(0,N)$ is a normal vector. Note that $(Y,Y)$ does not reduce the vanishing order nor increase the Clifford order; so by Lemma \ref{l:normalderivativesynchronous} it follows 
%\[o^t(\nabla_{(X,Y)}\sigma)\geq o^t(\sigma)-1\]
%and clearly
%\[o^t(c({(X,Y)})\sigma)\geq o^t(\sigma)-1\]
%and by induction for every differential operator $D\in \mathfrak{SDO}$
From Corollary \ref{c:generalderivativesynchronous}, we have
\[o^t(D\sigma)\geq o^t(\sigma)-o^f(D).\]
By definition, we have
\[o^t(D\sigma)\leq -o^{f}(D\sigma),\]
and consequently
\[o^{sc}(\sigma)\geq o^{t}(\sigma).\]
\end{proof}

\subsection{Rescaled module}
We are ready to  define a module over the algebra $\mathcal{A}(V,M)$, which will be later shown to be a module of sections of a certain vector bundle -- the rescaled bundle -- over the deformation space. 

The space of Laurent polynomials $\Gamma(V,E)[t,t^{-1}]$ is a module over the algebra $\mathcal{A}(V,M)$ and, hence, can be viewed as a sheaf over the deformation space $\DNC(V,M)$. We are interested in the following submodule of $\Gamma(V,E)[t,t^{-1}]$:

\begin{definition}\label{def:Rescaledmodule}
The subspace $\mathcal{S}(E,\nabla)\subset \Gamma(V,E)[t,t^{-1}]$ of the space of Laurent polynomials, consisting of the sections of the form
\[s=\sum_ps_pt^{-p}, \qquad\text{with}\quad o^{sc}(s_p)\geq p, 
\]
is called the {\em rescaled module}.
\end{definition}

The rescaled module is also a sheaf over $\DNC(V,M)$. In order to show that  $\mathcal{S}(E,\nabla)$  is isomorphic to a subspace of a space of section of a smooth bundle over $\DNC(V,M)$, we need to study the restriction  
\[\mathcal{S}_0(E,\nabla):=\mathcal{S}(E,\nabla)/t\mathcal{S}(E,\nabla)\]
of $\mathcal{S}(E,\nabla)$  to the zero fiber of $\DNC(V,M)$. We denote by $\mathcal{S}_{0,m}(E,\nabla)$ the quotient of $\mathcal{S}_0(E,\nabla)$ by the ideal of functions vanishing at $m\in M$.

Let $\mathcal{P(N)}\subset \Gamma(\mathcal{N},\bigoplus_{p=1}^qF^{p}/F^{p-1})$ denote the space of sections whose restriction to each fiber $\mathcal{N}_m$ of $\mathcal{N}$ are vector-valued polynomial functions 
     $\mathcal{N}_m\to \bigoplus_{p=1}^qF_m^{p}/F_m^{p-1}$.

\begin{lemma}
     Let $I_p(E)\subset\Gamma(V,E)$ denote the submodule of sections with scaling order at least $p$. There is a canonical isomorphism
     \[\mathcal{S}_0(E,\nabla)\buildrel\simeq\over\longrightarrow\bigoplus_pI_p(E)/I_{p+1}(E),\qquad\sum_ps_pt^{-p}\mapsto \sum_{p\geq 0}\langle s_p\rangle_p.\]
     Further, the map 
     \[\sum_ps_pt^{-p}\mapsto \big[\,X_m\mapsto\sum_p\frac{1}{p!}\langle\nabla_X^ps_p|_m\rangle_p\,\big],\]
     where $\langle.\rangle_p$ denotes the corresponding class in the quotient space $F_m^p/F_m^{p-1}$,
     defines an isomorphism between  $\mathcal{S}_{0,m}(E,\nabla)$  and the space 
     of vector-valued polynomial functions 
     $\mathcal{P(N)}$. 
     
     \qed
     
    % \cor{Shall we give a proof? Also $\mathcal{N}$  denotes both the normal bundle and the space of polynomial functions on it. I think it is rather confusing.}
\end{lemma}

\subsection{The construction of the rescaled bundle}
We now give a rundown of how to obtain a smooth vector bundle for which a ``big" subspace of the space of sections is isomorphic to $\mathcal{S}(E,\nabla)$.

As a set, this bundle is given by the union
\[\mathbb{E}= \Big(\bigoplus_{p=1}^qF^{p}/F^{p-1}\times\{0\}\Big)\bigsqcup_{t\neq 0}\big(E\times\{t\}\big).\]

To identify the elements of  $\mathcal{S}(E,\nabla)$ with sections of $\mathbb{E}$, we define for each element $\mu\in\DNC(V,M)$ an evaluation map $\varepsilon_\mu:\mathcal{S}(E,\nabla)\to \mathbb{E}$. Then the fiber $\mathbb{E}_{\mu}$ of $\mathbb{E}$ over $\mu$ is given by
\[\mathbb{E}_{\mu} =
\left\{
	\begin{array}{ll}
		E_{v}  & \mbox{if}\quad \mu=(v,t)\in V\times\bR^{\neq 0}, \\
		 \bigoplus_{p=1}^qF_m^{p}/F_m^{p-1}\qquad & \mbox{if}\quad \mu=(X_m,0)\in \mathcal{N}\times\{0\}.
	\end{array}
\right.\]

For $(v,\lambda)$, $v\in V$, $\lambda\neq0$, the evaluation  map is simply
\[\varepsilon_{(v,\lambda)}:\mathcal{S}(E,\nabla)\to E_v\]
\[\sum_ps_pt^{-p}\mapsto \sum_ps_p(v)\lambda^{-p}.\]
This map will be the evaluation map on the nonzero fiber element $(v,t)\in \textup{DNC}(V,M)$.

To define the evaluation over the zero fiber, we need a bit more work. 
First, note that for a vector field  $X\in\mathfrak{X}(V)$
we have a well-defined map 
\[t\nabla_X:\mathcal{S}(E,\nabla)\to \mathcal{S}(E,\nabla).\]
This map descends to the quotient space $\mathcal{S}_0(E,\nabla)$ as a locally nilpotent map:
\[t\nabla_X:\mathcal{S}_0(E,\nabla)\to \mathcal{S}_0(E,\nabla),\]
cf. Subsection~\ref{ss:alternativeXm}.
Hence we have a well-defined module homomorphism 
\[\exp(t\nabla_X):\mathcal{S}_0(E,\nabla)\to \mathcal{S}_0(E,\nabla).\]

 For $m\in M$ define the map
\[\varepsilon_{m}:\mathcal{S}(E,\nabla)\to \bigoplus_{p=1}^qF_m^{p}/F_m^{p-1}\]
\[\sum_ps_pt^{-p}\mapsto \sum_{p\leq 0}\langle s_p(m)\rangle_p.\]
And finally, for $X_m\in \mathcal{N}$
 define the map
\[\varepsilon_{X_m}:\mathcal{S}_0(E,\nabla)\to \bigoplus_{p=1}^qF_m^{p}/F_m^{p-1},\]
\[\varepsilon_{X_m}s=\varepsilon_m\big(\exp(t\nabla_X)s\big),\]
where $X$ is an arbitrary extension of $X_m$. The map $\varepsilon_{X_m}$ is the evaluation map over the zero fiber.

\subsection{The smooth structure on $\mathbb{E}$}
We now define a smooth structure on $\mathbb{E}$ by constructing a locally free sheaf on $\DNC(M,V)$ generated by  $\mathcal{S}(E,\nabla)$. This sheaf is the sheaf of smooth functions on $\mathbb{E}$.

\begin{definition}
Define the sheaf $\mathcal{E}$ on $\DNC(V,M)$ that consists of maps
\[\DNC(V,M)\to \Pi_{\mu\in \DNC(V,M)}\mathbb{E}_{\mu},\]
\[\mu\mapsto \tau(\mu)\]
such that for some $f_i\in C^{\infty}(\DNC(V,M))$ and $s_i\in \mathcal{S}(E,\nabla)$
\[\tau(\mu)=\sum_{i=1}^Nf_i(\mu)\varepsilon_{\mu}({s}_i)\]
for every $\mu\in \DNC(V,M)$.
\end{definition}
The elements of $\mathcal{E}$ are naturally identified with sections over $\mathbb{E}$.

\begin{theorem}\label{locallyfreesheafconstantrank}
The sheaf $\mathcal{E}$ is a locally free sheaf  of constant rank of modules over algebra $C^{\infty}(\DNC(V,M))$. Its rank is equal to the rank of the vector bundle $E\to V$. The evaluation map of the previous subsections defines a smooth structure on $\mathbb{E}$ and  identifies $\mathcal{E}$ with the space of smooth sections of $\mathbb{E}$.  This gives $\mathbb{E}$ the structure of a smooth bundle over $\DNC(V,M)$.  Fiber-wise $\mathbb{E}$ is given as follows
\[\begin{tikzcd}
\pi^*\textup{gr}(F)  \arrow[d, ]
&& E \arrow[d, ] \\
\mathcal{N}\times\{0\} &\bigsqcup_{t\neq 0}
& V\times\{t\}.
\end{tikzcd}\]
\end{theorem}

\begin{proof}
In some neighborhood of every $\mu \in \DNC(V,M)$, we show the sheaf is locally free of constant rank. 
%First, consider $\mu=X_m\in \mathcal{N}_m$. 
Fix a trivializing neighborhood \eqref{trivializingneighborhood} around $m\in M$ and consider the associated Euler-like vector field $\mathcal{R}$.
Take a local frame $\{e_i\}_{i=1}^{\textup{rank}(E)}$ for the restricted bundle $E|_M\to M$ such that each $e_i$ gives a local homogeneous section of degree $q_i=o^f(e_i)$ of the associated graded bundle
\[\textup{gr}(F):=\bigoplus_{p=1}^qF^p/F^{p-1}.\]
There are unique local $\mathcal{R}$-synchronous sections $\tilde{e}_i$ of $E$, so that $\tilde{e}_i|_M=e_i$. We show the set $\{\tilde{e}_it^{q_i}\}_{i=1}^{\textup{rank}(E)}$ generate the sheaf locally. Clearly in a neighborhood $\mu=(m,t)\in \DNC(V,M)$ these sections generate the sheaf. For $\mu=X_m\in\mathcal{N}_m$, we have
\begin{align*}
    \varepsilon_{X_m}(\tilde{e}_it^{q_i})&=\varepsilon_m\exp(t\nabla_X)\tilde{e}_it^{q_i}\\
    &=\varepsilon_m\sum_k\frac{t^{q_i+k}}{k!}\nabla_X^k\tilde{e}_i\\
    &=\langle e_i\rangle_{q_i}+\textup{higher degree terms}.
\end{align*}
Hence the sheaf is locally free of constant rank equal $\textup{rank}(E)$.

Now it remains to show the same for $\mu=(v,t)\in \DNC(V,M)$, with $v\notin M$. We may take a family of sections $\{s_i\}_i$ of $E$ forming a frame near $v$ and with supports away from $M$. In this case, the sections $s_i$ have scaling order $-\infty.$ So the set $\{s_i\}_i$  generates the sheaf  near $(v,t)$.
\end{proof}
\begin{definition}
The smooth vector bundle  $\mathbb{E}\to \DNC(V,M)$ is called the \emph{ rescaled bundle}.
\end{definition}

\subsection{Example: The spinor rescaled bundle}\label{regularrescaledbundle}
 The spinor rescaled bundle of Higson and Yi, \cite{Higson2019SpinorsAT},  was the main inspiration for this paper. We explain their setup briefly here.

    Assume $V=M\times M$ and $M\hookrightarrow M\times M$ is the diagonal embedding. If $\mathscr{E}\to M$ is a Clifford module with a Clifford connection $\nabla$, then $E=\mathscr{E}\boxtimes\mathscr{E}^*$ carries the connection $\nabla^{E}=\nabla\boxtimes 1+1\boxtimes\nabla$. Note that $E|_M$ is isomorphic to 
    
\begin{equation}\label{eq:cliffordcase}
    \mathbb{C}\textup{l}(TM)\otimes \textup{End}_{\textup{Cl}}(\mathscr{E})
\end{equation}
and it carries the Clifford filtration, which is compatible with the connection $\nabla^E$. Thus, we are in the situation of Subsection~\ref{ss:Scalingorder}, where the filtration is given by the Clifford filtration on the first factor of \eqref{eq:cliffordcase} (see Section~\ref{s:equivariantrescaledbundles} for more details on different filtrations in this case).

The curvature of this connection satisfies
    \[K^{E}=c\circ q\circ\gamma(R)\boxtimes1+1\boxtimes c^*\circ q\circ\gamma(R)+{F}^{\mathscr{E}/S}\boxtimes 1+1\boxtimes {F}^{\mathscr{E}^*/S}. \]
Here $R$ is the curvature of the Levi-Civita connection,  $\gamma:\mathfrak{so}(\mathscr{E})\xrightarrow[]{\simeq}\Lambda^2\mathscr{E}$ is the canonical isomorphism \cite[Section 2.2.10]{meinrenken2013clifford},  $q:\Lambda^*TM\to \textup{Cl}(TM)$
is the quantization \cite[Proposition 3.5]{berline1992heat} map, $c:\bC\textup{l}(TM)\to \textup{End}(\mathscr{E})$ is the Clifford action and $c^*:\mathbb{C}\textup{l}(TM)\to\textup{End}(\mathscr{E}^*)$ is the dual action. Thus the first summand in the above formula has Clifford filtration order 2, and, hence, when viewed as a differential operator, has Getzler filtration order 2.   Also, ${F}^{\mathscr{E}/S}$ and ${F}^{\mathscr{E}^*/S}$ are twisting curvatures of the Clifford module $\mathscr{E}$ and $\mathscr{E}^*$, respectively, cf. \cite[Prop.~3.43]{berline1992heat}. Recall that the twisting curvatures  commute with Clifford actions on $\mathscr{E}$ and $\mathscr{E}^*$. So, as a differential operator of order $0$, it has Getzler filtration order 0. We conclude that $K^E$ has Getzler order $2$. Therefore we obtain the rescaled bundle 

\[\mathbb{E}\to \DNC(M\times M,M)=\mathbb{T}M.\]
This rescaled bundle recovers the Getzler symbol calculus, \cite{Higson2019SpinorsAT}. This construction also recovers the local index formula, c.f \cite{Ludewig2020ASP,sadegh2021local}

The next section is dedicated to a generalization of the spinor rescaled bundle to the equivariant setting, from which we recover the Kirillov formula for the equivariant index.

\section{Equivariant Rescaled Bundles}\label{s:equivariantrescaledbundles}
Let $G$ be a compact Lie group acting on an oriented smooth manifold $M$
by orientation-preserving isometries. Consider a $\mathbb{Z}/2$-graded Clifford module $\mathscr{E}\to M$, which carries an even action of $G$. Let  $\nabla^{\mathscr{E}}$ be a $G$-equivariant connection on $\mathscr{E}$.

%%%
\subsection{The equivariant version of $\mathscr{E}\boxtimes\mathscr{E}^*$}\label{ss:equivariantbundle}

As equivariant version of the bundle $\mathscr{E}\boxtimes\mathscr{E}^*$ of subsection~\ref{regularrescaledbundle}, one would like to consider the bundle $\mathscr{E}\boxtimes\mathscr{E}^*\otimes\mathbb{C}[\mathfrak{g}]$. However, since this bundle has infinite dimension we prefer to work with its approximation 
\[
    E=\mathscr{E}\boxtimes\mathscr{E}^*\otimes\mathbb{C}[\mathfrak{g}]_{(J)},
\]
where $\mathbb{C}[\mathfrak{g}]_{(J)}$ is defined in subsection~\ref{ss:polynomialalgebra}.
This bundle is endowed with the natural connection  
\[\nabla:=\nabla^{\mathscr{E}}\boxtimes 1\otimes1+1\boxtimes\nabla^{\mathscr{E}^*}\otimes1+1\boxtimes1\otimes d\]
where $d$ is the trivial connection on the trivial bundle $\mathbb{C}[\mathfrak{g}]_{(J)}\times M\to M.$ The restricted bundle $F=E|_M$ is isomorphic to 
\[F\simeq\mathbb{C}\textup{l}(TM)\otimes \textup{End}_{\textup{Cl}}(\mathscr{E})\otimes\mathbb{C}[\mathfrak{g}]_{(J)}.\]
It carries the filtration $F^1\subset F^2\subset \cdots \subset F^q$ defined by
\[F^p:=\bigcup_{r\leq p/2}\mathbb{C}\textup{l}^{p-2r}(TM)\otimes \textup{End}_{\textup{Cl}}(\mathscr{E})\otimes\mathbb{C}[\mathfrak{g}]^{r}_{(J)}.\]
%where $\mathbb{C}[\mathfrak{g}]^{r}_{(J)}$ denotes the subspace of polynomials of degree at most $r$. 
The reason for this choice of filtration is that its associated graded space is isomorphic to
\[\textup{gr}(F)\simeq \Lambda(T^*M)\otimes \textup{End}_{\textup{Cl}}(\mathscr{E})\otimes\mathbb{C}[\mathfrak{g}]_{(J)}\]
with the grading as in \eqref{equivariantgradingofexterioralgebra}.

The bundle of endomorphisms $\textup{End}(E)\simeq \textup{End}(\mathscr{E})\boxtimes\textup{End}(\mathscr{E}^*)\otimes\textup{End}(\mathbb{C}[\mathfrak{g}]_{(J)})$ is filtered as follows
\[\textup{End}(E)^p=\bigcup_{r+s+2u=p}\textup{End}(\mathscr{E})^{r}\boxtimes\textup{End}(\mathscr{E}^*)^s\otimes\textup{End}(\mathbb{C}[\mathfrak{g}]_{(J)})^{u};\]
here we used the Clifford filtrations on $\textup{End}(\mathscr{E})$ and $\textup{End}(\mathscr{E}^*)$, and the filtration on $\textup{End}(\mathbb{C}[\mathfrak{g}]_{(J)})$ is induced from the filtration on $\mathbb{C}[\mathfrak{g}]_{(J)}.$

The connection $\nabla$ clearly satisfies conditions \eqref{technicalcondotionforrescaledbundle1} and \eqref{technicalcondotionforrescaledbundle2}. Hence, by Theorem \ref{locallyfreesheafconstantrank}, we obtain a rescaled bundle 

\[\mathbb{E}_J\to \mathbb{T}M.\]

\subsection{Operators on the rescaled bundles} 

Let $D$ be a differential operator of Getzler order $p$ (cf. Definition~\ref{def:Getzlerorder}) acting on $E\to M\times M$. Then we have the induced map

\[t^pD:\mathcal{S}(E,\nabla)\to\mathcal{S}(E,\nabla)\]
on the rescaled modules, which gives a smooth differential operator on the rescaled bundle:
\[\boldsymbol{D}\act\Gamma(\mathbb{E}_J).\]

\begin{definition}\label{getzlersymbold} We define the {\em Getzler symbol of $D$}, denoted by $\sigma^g(D)$, as the restriction of the operator $\boldsymbol{D}$ to the $t=0$ fiber of the tangent groupoid:
\[\sigma^g(D)\act \Gamma(TM, \pi^*\Lambda(T^*M)\otimes \textup{End}_{\textup{Cl}}(\mathscr{E})\otimes\mathbb{C}[\mathfrak{g}]_{(J)})\]
where $\pi:TM\to M$.
\end{definition}

\begin{example}\label{examplesofoperatorsonrescaledbundle}\phantom{}
\begin{itemize}

    \item For a vector field $\xi\in\mathfrak{X}(M)$, consider the differential operator $D=\nabla^{\mathscr{E}}_{\xi}$ acting on $\Gamma(M\times M,E)$ by differentiating along the first component of $M\times M.$ This operator has Getzler order $1$.
    
    \item For a vector field $\xi\in\mathfrak{X}(M)$, consider the operator $D=c(\xi)$ acting on $\Gamma(E)$ acting by left Clifford multiplication. This operator is of Getzler order $1$.
    \item For a polynomial $p(X)\in \mathbb{C}[\mathfrak{g}]$ of degree $k$, the differential operator $D=p(X)\act\Gamma({E})$ defined by multiplication by $p(x)$, has Getzler order  $2k.$

\end{itemize}

\end{example}

\subsection{A calculation of Getzler symbols of some operators}
We now  calculate the symbols of these operators. We need to fix  some conventions for the curvatures involved.
There is an Lie algebra isomorphism
\[\gamma:\mathfrak{so}(\mathscr{E})\xrightarrow[]{\simeq}\Lambda^2\mathscr{E}\]
\[T\mapsto \frac{1}{4}\sum_iT(e_i)\wedge e_i,\]
where the formula is given with respect to an orthonormal basis $\{e_i\}$ (see \cite[Section 2.2.10]{meinrenken2013clifford}).
The curvature of the bundle $\mathscr{E}$ satisfies
\[K^{\mathscr{E}}=c\circ q\circ\gamma(R)+{F}^{\mathscr{E}/S}\]
where $q:\Lambda^*TM\to \textup{Cl}(TM)$
is the quantization (see \cite[Proposition 3.5]{berline1992heat}) map and $c:\bC\textup{l}(TM)\to \textup{End}(\mathscr{E})$ is the Clifford action. We will use the notation
\[\mathsf{K}=\gamma(R)+\mathsf{F}^{\mathscr{E}/S}\in\mathcal{A}^2(M,\Lambda^2\mathscr{E}),\] 
where $\mathsf{F}^{\mathscr{E}/S}:=q^{-1}\circ c^{-1}({F}^{\mathscr{E}/S})$. %so that for every $\xi,\eta\in\mathfrak{X}(M)$, we obtain $\mathsf{K}(\xi,\eta)\in\Gamma(M,\Lambda^2\mathscr{E}).$ %We also denote
%\[\mathsf{K}^{\mathscr{E}}:=\mathscr{}q^{-1}\circ c^{-1}(\mathsf{F}^{\mathscr{E}/S})\in\mathcal{A}^2(M,\Lambda^2\mathscr{E}).\]

\begin{proposition}\label{p:zerofibervaluesofscaledops}
For $\xi\in\mathfrak{X}(M)$ and $p(X)\in\mathbb{C}_J$ of order $k\leq J$ we have 

\begin{itemize}
    \item the symbol $\sigma^g(\nabla^{\mathscr{E}}_{\xi})$ is given by
\[s\mapsto \Big[\eta\mapsto \partial_{\xi}s(\eta)+\frac{1}2\mathsf{K}(\eta,\xi)\wedge s(\eta)\Big],\]
    \item the symbol $\sigma^g(c({\xi}))$ is given by
    \[s\mapsto \Big[\eta\mapsto \xi\wedge s(\eta)\Big],\]
    
    \item the symbol $\sigma^g(p(X))$ is given by
    \[s\mapsto\Big[\eta\mapsto p(X)\cdot s(\eta)\Big].\]
\end{itemize}
Here these symbols are considered as operators acting on $\Gamma(T_mM,\Lambda T^*_mM\otimes\textup{End}_{\textup{Cl}}(\mathscr{E}_m)\otimes \bC[\mathfrak{g}]_{(J)})$ for $m\in M.$
\end{proposition}

\begin{proof}
The first bullet point is quite verbatim as the proof of \cite[Lemma 3.6.3]{Higson2019SpinorsAT}.
The second bullet point has the same proof as \cite[Lemma 3.6.2]{Higson2019SpinorsAT}.

To see the third bullet point, note that the operators ${\nabla^{\mathscr{E}}_{Y}}$ and ${p(X)}$ commute, therefore
\[\exp(t{\nabla^{\mathscr{E}}_{Y}}){t^{2k}p(X)}=t^{2k}{p(X)}\exp(t{\nabla^{\mathscr{E}}_{Y}})\]
and hence
\[\varepsilon_{Y_m}\Big({t^{2k}p(X)}s\Big)=t^{2k}p(X)\varepsilon_{(X_m,0)}(s)\]
from which the statement follows.

\end{proof}

\subsection{A conjugate connection}

By Proposition \ref{p:lichnerowicztypeformula}, up to lower order terms, the Bismut Laplacian $H(X)$ equals $\Big(\nabla^{\mathscr{E},\frac{1}4X}\Big)^*\nabla^{\mathscr{E},\frac{1}4X}$. Thus, it is related to the connection  $\nabla^{\mathscr{E},\frac{1}{4}X}$. But the rescaled bundle was defined using the connection 
\[\nabla^{\mathscr{E}}= \nabla^{\mathscr{E},\frac{1}4X} - 
\frac{1}{4}\theta_X.
\]
In particular, the computations of Proposition~\ref{p:zerofibervaluesofscaledops} help to compute the symbol of operators expressed in terms of the connection  $\nabla^{\mathscr{E}}$. Thus, it is convenient to express $H(X)$ in terms of this connection.  In other words, we want to express it in a coordinate system obtained by local trivializing $\mathscr{E}$ using the parallel transport of $\nabla^{\mathscr{E}}$ along geodesics (while naturally, it is expressed in trivialization obtained by parallel transport along $\nabla^{\mathscr{E},\frac{1}{4}X}$). One way of doing it is by conjugating $H(X)$ with the bundle map $\rho:\mathscr{E}\to \mathscr{E}$ which intertwines those two trivializations. This bundle map is given by the parallel transport along geodesics of the trivial line bundle with respect to the connection $d-\frac{1}{4}\theta_X$. This parallel transform formula is given as $\rho=e^{\alpha}$ where $\alpha:M\times M\to \mathfrak{g}^*$ is a smooth map on $M\times M$ with values in $\mathfrak{g}^*.$

Specifically, consider a smooth function
$\alpha:M\times M\to \mathfrak{g}^*\subset\bC[\mathfrak{g}]$
%$\alpha:M\times M\times\mathfrak{g}\to \bC^{\times}$
defined on the neighborhood of the diagonal
$\Delta M\hookrightarrow M\times M$
%$\Delta M\times\mathfrak{g}\hookrightarrow M\times M\times\mathfrak{g}$
by the formula
\[\alpha_X(\exp_m(\xi),m)=-\frac{1}{4}\int_0^1(\iota(\mathcal{R})\theta_X)\big(\exp(t\xi)\big)t^{-1}dt,\]
where $\mathcal{R}$ is the local Euler-like vector field on $M\times M$ defined by
\begin{equation}\label{eq:riemannianeulervectorfield}
    \mathcal{R}_{(\exp_m(\xi),m)}=\frac{d}{dt}|_{t=1}(\exp_m(t\xi),m).
\end{equation}

%for the embedding of the diagonal into $M\times M$, cf. Definition~\ref{def:Eulerlike}.
Away from a neighborhood of the diagonal, we can extend the function by zero using a cut-off function. We have $\mathcal{R}.\alpha_X(\exp_m(\xi),m)=-\frac{1}{4}\theta_X(\mathcal{R}_{(\exp_m(\xi),m)})$. Consider 
\[\omega_X:=-d_s\alpha_X-\frac{1}{4}\theta_X\in \Gamma\Big(\Lambda^1T_s^*(M\times M)\Big)\otimes\mathfrak{g}^*\]
that is a source-wise one-form and depends on $X$ linearly, and $d_s$ is the source-wise de Rham differential. We have the equality

\begin{itemize}
    \item 
    $\omega_X|_{\Delta M}=0$ and
    \item
    in a neighborhood of the diagonal
    \begin{equation}\label{omegaalongradialsvanishes}
        \iota(\mathcal{R})\omega_X=0.
    \end{equation}
\end{itemize}
From \eqref{omegaalongradialsvanishes}, we obtain
\[\mathscr{L}(\mathcal{R})\omega_X=-\frac{1}{4}\iota(\mathcal{R})d\theta_X\]
which means the first-order terms of Taylor expansion of $\omega$ are the same as $-\frac{1}{4}\iota(\mathcal{R})d\theta_X.$
\begin{lemma}\label{relationthetamoment}
For $\xi,\eta\in \mathfrak{X}(M)$ we have
\[d\theta_X(\xi,\eta)=-2(\mu^M(X)\xi,\eta).\]
So, in particular, $-\frac{1}{4}\iota(\mathcal{R})d\theta_X$ is the one-form dual to the vector field $\frac{1}{2}\mu(X)\mathcal{R}.$
\qed
\end{lemma}

So we have
\begin{equation}\label{firstderivativeofomega}
\omega_X=\frac{1}{4}\sum_{j}(\mu(X)\mathcal{R},\partial_j)dx^j+\mathcal{O}(2)
\end{equation}
where $\mathcal{O}(2)$ is a linear in $X$. 
\begin{lemma}\label{omegaxreducesscalingorderbyone}
For $\eta\in\mathfrak{X}(M)$ and $\sigma\in \Gamma(M\times M,\mathscr{E}\boxtimes\mathscr{E}^*)\otimes \bC[\mathfrak{g}]_{(J)},$we have
\[o^{\textup{sc}}(\omega_X(\eta)\sigma)\geq o^{\textup{sc}}(\sigma)-1.\]
\end{lemma}
\begin{proof}
Since $\omega_X$  vanishes along the diagonal $\Delta M\hookrightarrow M\times M$, the scaling order increase by at least $1$, and since it is linear in $X$ it increase the polynomial order by at most $1$. So the equivariant scaling (scaling order minus two times the polynomial order) order decreases by at most $1$.
\end{proof}
Therefore $t\omega_X(\eta)$ acts on the rescaled bundle and it gives corresponding operators on the quotient module and the rescaled bundle:
\[\boldsymbol{\omega_X(\eta)}:\mathcal{{S}}_{J,0}(E,\nabla)\to \mathcal{S}_{J,0}(E,\nabla)\]
and 
\[\boldsymbol{\omega_X(\eta)}:\Gamma(\bT M,\mathbb{E}_J)\to\Gamma(\bT M,\mathbb{E}_J).\]

\begin{lemma}\label{omegax=momentonquotient}
As an operator on $\mathcal{S}_{J,0}(E)$, we have
\[\boldsymbol{\omega_X(\eta)}\sigma=\frac{1}{4}(\mu(X)\mathcal{R},\eta)\cdot\sigma\]
\end{lemma}
\begin{proof}
This follows immediately from \eqref{firstderivativeofomega}.
%From \eqref{firstderivativeofomega}, we have
%\[\boldsymbol{\omega_X(\eta)}=\frac{1}{4}\sum_{j}(\mu(X)\mathcal{R},\eta)\]
\end{proof}

%where both sides are linear in $X\in\mathfrak{g}$. From this it follows that for $\xi,\eta\in \mathfrak{X}(M)$ 
\iffalse
\begin{equation}\label{firstderivativeofomega}
    \eta.\omega_X(\xi)=\frac{1}2(\mu(X)\xi,\eta)+\mathcal{O}(2).
\end{equation}
\fi

%From now on we consider $\alpha_X$ as function on the whole $M\times M$ (not just a neighborhood of diagonal) by extending it smoothly.

Consider the map
\begin{equation}\label{eq:rhomap}
    \rho:M\times M\to C^{\infty}(\mathfrak{g})
\end{equation}

\[\rho_X(m_1,m_2)=e^{\alpha_X(m_1,m_2)}.\]
 Define the conjugate connection 
 \[\rho_X\nabla^{\mathscr{E},X}\rho_X^{-1}\]
 and denote it by 
 \[\widehat{\nabla}^{X}.\]
 Note that $\widehat{\nabla}^X=\nabla^{\mathscr{E}}+\omega_X$. 
 \iffalse
 We denote the curvature of the connection $\widehat{\nabla}^{X}$ by
 \[\widehat{K}^X(\xi,\xi')=\widehat{\nabla}^{X}_{\xi}\widehat{\nabla}^{X}_{\xi'}-\widehat{\nabla}^{X}_{\xi'}\widehat{\nabla}^{X}_{\xi}-\widehat{\nabla}^{X}_{[\xi,\xi']}.\]
 A direct calculation shows
 \begin{equation}\label{hatcurvatureformula}
     \widehat{K}^X(\xi,\xi')=K^{\mathscr{E}}(\xi,\xi')+\frac{1}{2}\big(\mu^M(X)\xi,\xi'\big).
 \end{equation}
 
 %\begin{lemma}
 %We have equalities
 % \begin{itemize}
   %  \item $\widehat{\nabla}^{X}=\nabla^{\mathscr{E}}+\omega_X.$
   %  \item $\widehat{K}^X(\xi,\xi')=K(\xi,\xi')+\frac{1}{2}\big(\mu^M(X)\xi,\xi'\big)$
 %\end{itemize}
 %\end{lemma}

\begin{proposition}\label{nablahatscalingorderis-1}
For $\sigma\in \Gamma(M\times M,\mathscr{E}\boxtimes \mathscr{E}^*)\otimes \bC[\mathfrak{g}]_{(J)}$ and $\eta\in \mathfrak{X}(M)$ we have
 $$o^{sc}_{eq}(\widehat{\nabla}^{X}_{\eta}\sigma)\geq o^{sc}_{eq}(\sigma)-1.$$

\end{proposition}

\begin{proof}
We have
\[\widehat{\nabla}^{X}_{\eta}\sigma=\nabla^{\mathscr{E}}_{\eta}\sigma +\omega_X(\eta)\]
where $\omega_X(\eta)$  is linear polynomials of $X$ that vanish along the diagonal $\Delta M\hookrightarrow M\times M.$ Hence this two functions increase the order of vanishing by at least 1 and increase the polynomial order by 1. Hence the equivariant scaling order decreases by at most $1$ (by increasing the scaling order by at least $1$  minus $2$ for polynomial order).
\end{proof}
\fi
By Lemma \ref{omegax=momentonquotient}, the scaled operator $t\widehat{\nabla}^{X}_{\eta}$ acts on the rescaled module and hence on the rescaled bundle. From Proposition \ref{p:zerofibervaluesofscaledops} it follows:

\begin{proposition}\label{p:getzlersymbolofconjugateconnection}
The Getzler's symbol of $\widehat{\nabla}^{X}_{\xi}$ is given for $m\in M$ by
%The operator $t\widehat{\nabla}^X_{\eta}$ on the rescaled bundle $\slashed{S}_J$ comprised of source-wise differential operators that on the fibers $\bT M_{(m,0)}$, for $t=0$ and $m\in M$  

\[\sigma^g(\widehat{\nabla}^{X}_{\xi})\act \Gamma(T _{m}M,\Lambda T^*_mM\otimes\textup{End}_{\textup{Cl}}(\mathscr{E}_m)\otimes \bC[\mathfrak{g}]_{(J)})\]
\[s\mapsto \Big[\eta\mapsto \partial_{\xi}s(\eta)+\frac{1}2\mathsf{K}(\eta,\xi)\wedge s(\mu)+\frac{1}{4}(\mu(X)\eta,\xi)s(\xi)\Big].\]
\qed
\end{proposition}

 Denote by $\widehat{H}(X)$ the conjugate operator $\rho_X H(X)\rho_X^{-1}$.
 From Propositions \ref{p:lichnerowicztypeformula} and \ref{p:getzlersymbolofconjugateconnection} it follows that the Getzler symbol of the operator $\widehat{H}(X)$, in a local orthonormal frame $\{e_i\}_i$,
 is given by the formula
 
 \begin{equation}\label{symbolofconjugateofbismutlaplacian}
    \sigma^g(\widehat{H}(X))=-\sum_i\Big(\partial_i-\frac{1}4\sum_j\big( R_{ij}+\mu^M_{ij}(X)\big)x^j\Big)^2+\mathsf{F}^{\mathscr{E}/S}+\mu^{\mathscr{E}/S}(X),
\end{equation}
 where $R_{ij}=(R\partial_i,\partial_j)$.
 We will use the notation $\widehat{\mathcal{H}}(X):=\sigma^g(\widehat{H}(X))$ for the symbol. So in particular, the family of operators
 $t^2\widehat{H}(X)$ acting on the rescaled bundle $\Gamma(\mathbb{E}_J)$ over the $t$-fibers of the tangent groupoid along with the operator $\widehat{\mathcal{H}}(X)$ acting on the rescaled bundle over the $0$-fibers give a smooth differential operator acting on $\Gamma(\mathbb{E}_J)$.
 % which extends smoothly to zero fiber as and operator acting on the space
%\[\Gamma(TM,\pi^*\Lambda^*T^*M\otimes\textup{End}_{\textup{Cl}}(\mathscr{E}_m)\otimes \bC[\mathfrak{g}]_{(J)}))\]
%which in a local orthonormal frame is given by the formula $\{e_i\}_i$
%%  \widehat{\mathsf{H}}(X)=-\sum_i\Big(\partial_i-\frac{1}4\sum_j\big( R_{ij}+\mu^M_{ij}(X)\big)x^j\Big)^2+\mathsf{F}^{\mathscr{E}/S}+\mu^{\mathscr{E}/S}(X)
%\end{equation}
 %and $$\mathsf{F}^{\mathscr{E}/S}(e_i,e_j)=\frac{1}{4}\sum_{ij}(F^{\mathscr{E}/S}(e_i,e_j)e_k,e_l)dx^kdx^l.$$

The operator \eqref{symbolofconjugateofbismutlaplacian} is a \emph{generalized harmonic oscillator operator} (cf. \cite[Definition 4.11]{berline1992heat}) and the Mehler kernel associated to the heat operator $e^{-\tau \widehat{\mathcal{H}}(X)}$, $\tau>0$, is given by the formula

\begin{equation*}
     {K}_{\tau}^{\widehat{\mathcal{H}}}(X,\xi)=(4\pi\tau)^{-n/2}\textup{det}^{1/2}\Big(\frac{\tau (R+\mu^M(X))/2}{\sinh \tau (R+\mu^M(X))/2}\Big)
     \exp(-\tau(\mathsf{F}^{\mathscr{E}/S}+\mu^{\mathscr{E}/S}(X)))
\end{equation*}
\begin{equation*}
   \exp(-\frac{1}{4\tau}\langle\xi|\tau (R+\mu^M(X))/2\coth\tau(R+\mu^M(X))/2|\xi\rangle),
\end{equation*}
which can be written in term of equivariant curvature terms as 
\begin{equation}\label{mehlerskerneloftheconjugatesymbol}
    (4\pi \tau)^{-n/2}\textup{det}^{1/2}\Big(\frac{\tau R_{\mathfrak{g}}(X)/2}{\sinh \tau R_{\mathfrak{g}}(X)/2}\Big)\exp(-\frac{1}{4\tau}\langle\xi|\tau R_{\mathfrak{g}}/2\coth \tau R_{\mathfrak{g}}/2|\xi\rangle)\exp(-\mathsf{F}^{\mathscr{E}/S}_{\mathfrak{g}}(X)).
\end{equation}

%subsection 
\section{The Heat Kernel Asymptotics and Proof of the Kirillov Formula}\label{s:hearkernelasymptoticsproffkirillovformula}

We denote by $K_{\tau}(x,y,X)\in \Gamma(M\times M,\mathscr{E}\boxtimes\mathscr{E}^*)$ the heat kernel of $e^{-\tau \widehat{H}(X)}$ where $\widehat{H}(X)$ is the conjugate operator $\rho_XH(X)\rho_X^{-1}$ and $H(X)=H_{\frac{1}{4}}(X)$ is the Bismut Laplacian \eqref{bismutlaplacianu=1/4}.
%We will use the notation $\hat{H}(X)$ for the conjugate operator $\rho_X H(X)\rho_X^{-1}.$
\begin{lemma}\label{asymptoticsofbismutlaplacian}\cite[Theorem 7.15]{roe1999elliptic}
     Fix $X\in\mathfrak{g}$. For $J\geq0$, there exists an asymptotic expansion
     \[K_{\tau}(x,y,X)\sim(4\pi t)^{-n/2}\exp\Big(-\frac{d^2(x,y)}{4t}\Big)\,\sum_{j\geq 0}\tau^j\Phi_{j,J}(x,y,X)\]
     where $\Phi_{j,J}(x,y,X)\in \Gamma(M\times M, \mathscr{E}\boxtimes\mathscr{E}^*)$. In a neighborhood of the diagonal, we have the recurrence relations (cf. (7.17) in  \cite{roe1999elliptic})
\begin{equation}\label{recurrenceoftheasymptoticcoefficients}
    (\nabla_{\mathcal{R}}+j+\nabla_{\mathcal{R}}\log g^{1/4})\Phi_{j,J}=-\widehat{H}(X)\Phi_{j-1,J}
\end{equation}
where $\mathcal{R}$ is the  Euler-like vector filed \eqref{eq:riemannianeulervectorfield}, $\Phi_{0,J}(x,y,X):\mathscr{E}_y\to\mathscr{E}_x$ is the parallel transport along the geodesic between $x$ and $y$ with respect to the connection $\nabla^{\mathscr{E}}$. Furthermore, $$g=\textup{det}(g_{ij})$$ where $g_{ij}(x,y)$ are functions defined in a neighborhood of the diagonal on $M\times M$ whose restriction to the source fibers $M\times\{y\}$ are the components of the metric in the normal coordinate at $y$. \qed
\end{lemma}

\begin{remark}\label{r:polynomialdegreeofasymptoticcoefficients}
From the recurrence relations \eqref{recurrenceoftheasymptoticcoefficients}, it follows that the coefficients $\Phi_{j,J}(x,y,X)$ are polynomials in variable $X$. Hence we have
\[\Phi_{j,J}(x,y,X)\in \Gamma(M\times M, \mathscr{E}\boxtimes\mathscr{E}^*\otimes\mathbb{C}[\mathfrak{g}]_{(J)}).\]
\end{remark}

For every $N\geq0$, we will denote by 
\[K_{\tau}^N:=(4\pi \tau)^{-n/2}e^{-\frac{d^2(x,y)}{4\tau}}\sum^N_{j\geq 0}\tau^j\Phi_{j,J}(x,y,X)\]
the approximate heat kernel.
Consider the rescaled heat kernel $\sigma_{\tau}^N(x,y,X):=t^n\psi(d(x,y))K_{t^2\tau}^N$, where $\psi:\bR\to\bR^{\geq0}$ is the cut-off functions such that $\psi|_{[-r/2,r/2]}=1$ and $\psi$ vanishes outside $[-r,r]$ where $r$ is the injectivity radius of $M$.

\begin{proposition}\label{scalingorderofasymptoticcoefficients}
For every $j\geq0$, we have
\[o^{\textup{sc}}(\Phi_{j,J})\geq -2j.\]
\end{proposition}
\begin{proof}
We use induction on $j$.
Note that for $j=0$ the case is clear by Remark \ref{r:polynomialdegreeofasymptoticcoefficients} and Theorem \ref{t:taylororderequalscalingorder}, since $\Phi_{0,J}$ is the parallel transport map and hence synchronous. 
We show for $j>0$
\begin{equation}\label{eq:scalingorderofasymptoticcoefficients}
o^{\textup{sc}}(\Phi_{j,J})=o^{\textup{sc}}\big((\nabla_{\mathcal{R}}+j+\nabla_{\mathcal{R}}\log g^{1/4})\Phi_{j,J}\big)    
\end{equation}
from which we then deduce
\[o^{\textup{sc}}(\Phi_{j,J})\geq o^{\textup{sc}}(-\widehat{H}(X)\Phi_{j-1,J})\geq -2(j-1)-2=-2j.\]

So we just need to prove \eqref{eq:scalingorderofasymptoticcoefficients}. Using Theorem \ref{t:taylororderequalscalingorder}, we only need to show the equality in a trivializing neighborhood associated with the Euler-like vector field $\mathcal{R}$. So if the Taylor expansion of $\Phi_{i,J}$ is given by
\[\Phi_{j,J}\sim \sum_{I} y^{I}\sigma_{I},\] then we have the Taylor expansion
\[(\nabla_{\mathcal{R}}+j)\Phi_{j,J}\sim \sum_{I} (|I|+j)y^{I}\sigma_{I}\]
which has the same scaling order (=Taylor order) as $\Phi_{j,J}$. 
The Euler-like vector field $\mathcal{R}$ vanishes along the diagonal; hence so does $\nabla_{\mathcal{R}}\log g^{1/4}$. Thus this term has no effect on the scaling order, and therefore, the equality \eqref{eq:scalingorderofasymptoticcoefficients} follows.

%Using induction, for $i>0$, from the equation \eqref{recurrenceoftheasymptoticcoefficients} and the fact that $\mathcal{R}$ vanishes along the diagonal we deduce
\end{proof}

Note that for $\tau>0$ the we may consider $\psi(d(x,y))e^{-\frac{d^2(x,y)}{4\tau t^2}}$  as a smooth function on $\bT M.$ By Proposition \ref{scalingorderofasymptoticcoefficients}, we have 
\[\sum^N_{j\geq 0}\tau^{j}t^{2j}\Phi_{j,J}\in\mathcal{S}_J(E,\nabla)\]
and hence $\sigma_{\tau}^N(x,y,X)\in\Gamma(\bT M,\mathbb{E}_J).$

The following lemma is a generalization of Lemma 2.8 in \cite{Ludewig2020ASP}. Consider a section $s\in \Gamma(M\times M\times \bR,\mathscr{E}\boxtimes\mathscr{E}^*)$ and the natural projection $\pi:\bT M\to M\times M\times\bR$. 

\begin{lemma}\label{l:zelinlemma}
For  $N\geq n$, the section $t^{N+1}s$, after composing with $\pi:\bT M\to M\times M\times\bR,$ may be considered as a smooth section of the rescaled bundle $\mathbb{E}_J\to \bT M$ that vanishes to order $2N-n+1$ on the zero fiber of the tangent groupoid. 
\end{lemma}

\begin{proof}
Since the function $t:\bT M\to\bR$ is the canonical projection. This projection, as a function, vanishes on the zero fiber. Since we may write
\[t^{2N+1}s =t^{2N-n}\cdot t^{n+1}s,\]
it is enough to show that $t^ns$ gives a smooth section of the rescaled bundle that vanishes along the zero fiber. Note that we may find sections $\sigma_1,\cdots,\sigma_p\in \Gamma(M\times M,\mathscr{E}\boxtimes\mathscr{E}^*)$ and functions $f_1,\cdots,f_p\in C^{\infty}(M\times M\times\bR,\mathscr{E}\boxtimes\mathscr{E}^*)$ such that
\[s =\sum_{j=1}^pf_j(x,y,t)\sigma_j(x,y).\]
Since $t^{n+1}\sigma_{j}$ is an element of the rescaled module that gives zero values on the zero fiber of the tangent groupoid and $f_i\circ\pi\in C^{\infty}(\bT M)$, we obtain a smooth section of the rescaled bundle of the form
\[\sum_{j=j}^pf_i\circ\pi\cdot t^{n+1}\sigma_j\]
that vanishes along the zero fiber of $\bT M$.
\end{proof}

From Lemma \ref{asymptoticsofbismutlaplacian}, we have
\[K_{\tau t^2}-\sigma_{\tau}^N=\mathcal{O}(t^{2N+1})\]
where $\mathcal{O}(t^{2N+1})\in\Gamma(M\times M\times\bR,\mathscr{E}\boxtimes\mathscr{E}^*)$ and if $N>n/2$, by Lemma \ref{l:zelinlemma}, it gives a smooth section of the rescaled bundle vanishing over the zero fiber of the tangent groupoid. In particular, we have 
\begin{theorem}
The rescaled heat kernel $\boldsymbol{K}_{\tau}:=K_{\tau t^2}$ gives a smooth section of the rescaled bundle $\mathbb{E}_J\to\bT M$, which over the zero fiber of the tangent groupoid agrees with values of sections $\sigma_{\tau}^N$ for $N>n/2.$ \qed
\end{theorem}
The section $\boldsymbol{K}_{\tau}\in \Gamma(\bT M,\mathbb{E}_J)$ satisfies the heat equation
\[(\partial_{\tau}+\boldsymbol{\widehat{H}}(X))\boldsymbol{K}_{\tau}=0\]
where $\boldsymbol{\widehat{H}}(X)$ is the smooth extension of the operator $t^2\widehat{H}(X)$ to zero fiber of the tangent groupoid.
Note that on zero fiber of the tangent groupoid, $\boldsymbol{\widehat{H}}(X)$ acts as the harmonic oscillator operator $\widehat{\mathcal{H}}(X)$ given in \eqref{symbolofconjugateofbismutlaplacian}. Therefore we have
\[(\partial_{\tau}+\widehat{\mathcal{H}}(X))(\boldsymbol{K_{\tau}|_{t=0}})=0.\]
%if $N>n/2.$
The uniqueness property of the Mehler kernel then proves 
\begin{equation}\label{rescaledheatkernelisharmonicattequaltozero}
    \boldsymbol{K}|_{t=0}={K}_{\tau}^{\widehat{\mathcal{H}}}
\end{equation}
where ${K}_{\tau}^{\widehat{\mathcal{H}}}$ is as in \eqref{mehlerskerneloftheconjugatesymbol}.

%subsection
\subsection{Supertraces on the rescaled bundle}
Let $M\times\bR\hookrightarrow \bT M$ denote the embedding of the unit space of the tangent groupoid. 

%Note that $\bT M$ is a groupoid with the object space . 
\begin{lemma}\label{strreductiontoobjectspace}
     The following map is well-defined
     \[\mathfrak{s}:\Gamma(\bT M,\mathbb{E}_J)\to C^{\infty}(M\times \bR,\mathbb{C}[\mathfrak{g}]_{(J)})\]

\[\sigma \mapsto\left\{
	\begin{array}{ll}
	 (m,t)   \mapsto t^{-n}\str{(\sigma(m,m,t))} & t\neq0 \\
	
	(m,0)  \mapsto \ \str{(\sigma(0_m,0))} & t=0
	\end{array}
\right.\]
Here, $C^{\infty}(M\times \bR,\mathbb{C}[\mathfrak{g}]_{(J)})$ denotes the $\mathbb{C}[\mathfrak{g}]_{(J)}$-valued smooth functions on $M\times\mathbb{R}$. The supertrace $\textup{str}(\sigma(m,m,t))$ is induced by the supertrace on $\mathscr{E}\otimes\mathscr{E}^*$ and the supertrace $\textup{str}(\sigma(0_m,0))$ is given by the Berezin integral (cf. \cite[Page~40]{berline1992heat}). 

     %\[\sigma\mapsto \textup{str}\big(i^*(\sigma)\big)\]
\end{lemma}
\begin{proof}
Since every $\sigma\in\Gamma(\bT M,\mathbb{S})$ can be written as a sum 
\[\sigma=\sum_kf_k\sigma_k\]
where $f_k\in C^{\infty}(\bT M)$ and $\sigma_k\in \mathcal{S}_J(E,\nabla),$
it is enough to prove the lemma for $\sigma\in\mathcal{S}_J(E,\nabla)$. So assume $\sigma=\sum_ks_kt^{-k}$
where $s_k\in \Gamma(M\times M,\mathscr{E}\boxtimes\mathscr{E}^*\otimes\bC[\mathfrak{g}]_{(J)})$ is of scaling order $k$ or more. 
We claim $\textup{str}(s_k(m,m))=0$ for $k>-n$. To see why, note that by definition
%\[o^{\textup{sc}}(s_k)\leq o^{\textup{sc}}(s_k)\]
%and 
\[o^{\textup{sc}}(s_k)\leq -o^{\textup{f}}(s_k).\]
So when $k>-n$,
\[o^{\textup{f}}(s_k)< n\]
and therefore $\textup{str}(s_k(m,m))=0.$
Also, for $k<-n$, $t^{-n}\textup{str}(s_k)(x,x)t^{-k}$ extends smoothly to zero at $t=0$, since $k-n>0.$ For $k=-n$ the function 
$t^{-n}\textup{str}(s_{-n}(m,m))t^n$ is constant in $t$. Therefore
\[t^{-n}\textup{str}(\sigma(m,m,t))\]
extends smoothly to 
\[\textup{str}(s_{-n}(m,m))\]
at $t=0$ which equals $\textup{str}(\sigma(0_m,0)).$
\end{proof}

Using Lemma \ref{strreductiontoobjectspace}, we define a family of functionals on $\Gamma(\bT M,\mathbb{E}_J)$:
\[\Str_t:\Gamma(\bT M,\mathbb{E}_J)\to\bC\]
given as 
\begin{equation}\label{familysupertraces}
    \Str_t(\sigma)=\int_M\mathfrak{s}(\sigma)(m,t)dm
\end{equation}
and obviously 
$\Str_t(\sigma)\to \Str_0(\sigma).$

\begin{theorem}[McKean-Singer]\label{usualmckeansinger} For $t\neq0$ and $\tau>0$
\[\Str_t(\boldsymbol{K}_{\tau})\]
is independent of $t, \tau$ and by Proposition \ref{mckeansingergeneralization} equals the equivariant index 
\[\textup{ind}(e^{-X},\slashed{D}).\] 
\end{theorem}
Since the supertraces $\textup{Str}_t$ \eqref{familysupertraces}, are defined for every $t$, Theorem \ref{usualmckeansinger} also applies for $t=0$. So by evaluating at $\tau=1$ and $t=0$ and using \eqref{rescaledheatkernelisharmonicattequaltozero} we obtain 

\begin{theorem}[Kirillov Formula] The equivariant index is given by the integral of equivariant differential forms as follows:
\[\textup{ind}(e^{-X},\slashed{D})=(2\pi i)^{-n/2}\int_M\Ahat_{\mathfrak{g}}(X,M)\Ch_{\mathfrak{g}}(X,\mathscr{E}/S).\]

\end{theorem}

%%%%%%%%%%%%%%%%%%%%%%%%%%
\section{The Other Examples of the Rescaled Bundles}\label{examplesofrescaledbundles}

We conclude with several more examples and applications of our construction of the rescaled bundle in different geometric situations. 

%%%%%%%%%%%%%
\subsection{The case with no filtration}\label{ss:nofiltration}
Consider an embedding $M\hookrightarrow V$.
For every vector bundle $E\to V$ with trivial filtrations on $E|_M$ and $\textup{End}(E)$ and for every connection $\nabla$ on $E\to V$ we obtain a rescaled bundle 
\[\mathbb{E}\to \textup{DNC}(V,M).\]
Since in this case, the scaling order coincides with the vanishing order, this bundle is independent of the choice of the connection $\nabla$. Indeed the rescaled bundle $\mathbb{E}$ is isomorphic to the pullback bundle $\pi_V^*E$ where $\pi_V:\DNC(V,M)\to V$ is the canonical projection. 

One important case of this situation is the description of the Witten deformation via the deformation to normal cone construction, recently obtained by Omar Mohsen, \cite{Mohsen22}. In this case $E=\Lambda^*T^*V\to V$. In the next subsection, we discuss an equivariant generalization of  Mohsen's construction.

%%%%%%%%%%%
\subsection{The equivariant Witten and Novikov deformations}\label{ex:triviallyfilteredcase} 

Let $G$ be a compact Lie group action on a closed manifold $V$. Then  $\Lambda^*T^*V\to V$ is a $G$-equivariant bundle  induced with the Levi-Civita connection $\nabla^{LC}$. As in subsection~\ref{ss:equivariantbundle}, we consider the equivariant version of this bundle 
\[
    E:= \Lambda^*T^*V\otimes \mathbb{C}[\frak{g}]_{(J)}
\]
and endow it with the connection 
\[\nabla:=\nabla^{LC}\otimes1+1\otimes d\]

Let $f:V\to \mathbb{R}$ be a $G$-equivariant Morse-Bott function, with the critical submanifold $M\hookrightarrow V$. We consider the filtrations on  $E|_M$ and on $\textup{E}$ defined by the trivial filtrations on $\Lambda^*T^*V|_M$ and $\Lambda^*T^*V$ and the filtration by $2\deg P$ on $\mathbb{C}[\frak{g}]_{(J)}$. (Note that this is different from the filtration \eqref{equivariantgradingofexterioralgebra}, where we also consider a non-trivial filtration on $\Lambda^*T^*V$). Then one obtains a rescaled bundle  $\mathbb{E}\to \textup{DNC}(V,M)$. The arguments of \cite{Mohsen22} extend naturally to the equivariant case and show that the ``rescaled Witten deformation'' of the de Rham-Dirac operator
\[t\,\big(\,e^{\frac{1}{t^2}f}de^{-\frac{1}{t^2}f}+e^{-\frac{1}{t^2}f}d^*e^{\frac{1}{t^2}f}\,\big)= t\,\big(\,d+d^*+ \frac1{t^2}c(df)\,\big)\]
acts on the rescaled module $\mathcal{S}(E,\nabla)$ and hence
extends to a smooth $G$-equivariant differential operator acting on the section of the rescaled bundle $\Gamma(\mathbb{E}).$
One can also generalize it to the Novikov deformation, defined by a closed differential form $\omega$:
\[
    t\,\big(\,d+d^*+ \frac1{t^2}c(\omega)\,\big),
\]
cf. \cite{Pazhitnov87,BrFar1,BrFar2} for the non-equivariant case and \cite{BrFar3,BrFar4} for the equivariant case.

\subsection{The fixed point formula for the equivariant index}\label{relativerescaledbundle}

Another example of the rescaled bundle arises in the context of the equivariant index formula which will appear in a joint paper of the second author with Yiannis Loizides, Jesus Sanchez and Shiqi Liu:

Consider a compact Lie group $G$ acting isometrically on $M$. For an element $g\in G$, denote by $M^g$ the fixed submanifold under the action of $g$. Associated to the diagonal embedding $M^g\hookrightarrow M\times M$, the deformation space $\DNC(M\times M,M^g)$ is called the \emph{relative tangent groupoid} and denoted by $\mathbb{T}_gM$.

Assume $\mathscr{E}\to M$ is a $G$-equivariant Clifford module that carries a $G$-invariant connection $\nabla$. The bundle $E=\mathscr{E}\boxtimes\mathscr{E}^*\to M\times M$ carries the induced connection $\nabla^E$. Using a similar argument as in Example \ref{regularrescaledbundle}, we obtain a rescaled bundle 
\[\mathbb{E}_g\to \mathbb{T}_gM.\]
Using this vector bundle, one may recover the equivariant index formula.

%I haven't check this yet but I think for the Bismut's Laplacian
%\[H(X)=(\slashed{D}+\frac{1}{4}c(X))^2+\mathcal{L}(X)\]
%with $X\in\mathfrak{g}$, we have
%\begin{theorem}
%The rescaled operators $t^2H(X)$ on the rescaled bundle 
%\[\slashed{S}_k\to \bT M\]
%for $k\geq 2$, extend smoothly to the zero fiber given as the harmonic oscillator operator
%\[-\sum_i(\partial_i-\frac{1}{4}\sum_j(R_{ij}+\mu^M_{ij}(X)x^j))^2.\]

%\end{theorem}

\bibliographystyle{amsplain}

\providecommand{\bysame}{\leavevmode\hbox to3em{\hrulefill}\thinspace}
\providecommand{\MR}{\relax\ifhmode\unskip\space\fi MR }
\providecommand{\MRhref}[2]{%
  \href{http://www.ams.org/mathscinet-getitem?mr=#1}{#2}
}
\providecommand{\href}[2]{#2}

%\begin{thebibliography}{10}

\bibliography{references}
% I prefer to use the IEEE bibliography style. 
% That's  NOT required by the NSF guidelines. 
% Feel Free to use whatever style you prefer
%\bibliographystyle{apalike}

\end{document}